\documentclass[11pt, reqno]{amsart}
\usepackage{amsmath,amssymb,amsthm,amscd, amsxtra, amsfonts,graphicx,fancybox,cancel,framed, mathrsfs, stmaryrd, cite }
\usepackage{tabu}

\usepackage[all]{xy}
\usepackage{amssymb}
\usepackage{enumitem}
\usepackage{latexsym}
\usepackage{amscd}
\usepackage{color}
\usepackage{tikz}
\usepackage{verbatim}
\usepackage{esint}

\input amssym.def
\input amssym

\theoremstyle{plain}
\newtheorem{theorem}{Theorem}[section]
\newtheorem*{theorem*}{Theorem}
\newtheorem{corollary}[theorem]{Corollary}
\newtheorem{conjecture}[theorem]{Conjecture}
\newtheorem{lemma}[theorem]{Lemma}
\newtheorem{proposition}[theorem]{Proposition}
\newtheorem*{conjecture*}{Conjecture}

\theoremstyle{definition}

\theoremstyle{remark}
\newtheorem*{remark}{Remark}
\newtheorem*{remarks}{Remarks}

\numberwithin{equation}{section}

\newcommand{\bea}{\begin{eqnarray}}
\newcommand{\eea}{\end{eqnarray}}
\newcommand{\be}{\begin {equation}}
\newcommand{\ee}{\end{equation}}

\newcommand{\Z}{\mathbb Z}

\newcommand{\N}{\mathbb N}
\newcommand{\C}{\mathbb C}
\newcommand{\Q}{\Bbb Q}

\newcommand{\SL}{\operatorname{SL}}

\renewcommand{\pmod}[1]{\  \,  \left(  \mathrm{mod} \,  #1 \right)}

\renewcommand{\b}[1]{{\boldsymbol{#1}}}
\newcommand{\bb}[1]{{\boldsymbol{#1}}}

\newcommand{\sgn}{{\rm sgn}}

\usepackage{amssymb}

\allowdisplaybreaks

\makeatletter
\newcommand{\vast}{\bBigg@{4}}
\newcommand{\Vast}{\bBigg@{4.5}}
\makeatother

\renewcommand{\pmod}[1]{\  \,  \left(  \mathrm{mod} \,  #1 \right)}

\begin{document}

\title[ ]{Quantum modular forms and plumbing graphs of $3$-manifolds}
\author{Kathrin Bringmann, Karl Mahlburg, Antun Milas}

\address{Mathematical Institute, University of Cologne, Weyertal 86-90, 50931 Cologne, Germany}
\email{kbringma@math.uni-koeln.de}

\address{Department of Mathematics, Louisiana State University, Baton Rouge, LA 70803, USA}
\email{mahlburg@math.lsu.edu}

\address{Department of Mathematics and Statistics, SUNY-Albany, Albany, NY 12222, U.S.A.}
\email{amilas@albany.edu}

\thanks{The research of the first author is supported by the Alfried Krupp Prize for Young University Teachers of the Krupp foundation and by the Collaborative Research Centre / Transregio on Symplectic Structures in Geometry, Algebra and Dynamics (CRC/TRR 191) of the German Research Foundation. The third author was partially supported by NSF Grant DMS-1601070.}



\maketitle

\begin{abstract} In this paper, we study quantum modular forms in connection to quantum invariants of plumbed 3-manifolds introduced recently by Gukov, Pei,  Putrov, and Vafa.
We explicitly compute these invariants for any $3$-leg star plumbing graphs whose associated matrix is unimodular and positive definite.
For these  graphs we confirm a quantum modularity conjecture of Gukov.
We also analyze the invariants for general $n$-leg star graphs with unimodular plumbing matrices, and prove that they can be expressed as linear combinations of quantum modular forms.

\end{abstract}

\section{Introduction and statement of results}

\subsection{Introduction}

Quantum invariants are important numerical invariants  of knots and 3-manifolds and are originally defined in \cite{W} using tools of quantum field
theory. Witten conjectured the existence of topological
invariants of 3-manifolds generalizing the Jones polynomial to links in arbitrary closed oriented 3-manifolds.
Using a modular tensor category coming from the quantum group $\text{U}_q(\frak{sl}_2)$ at roots of unity, Reshetikhin and Turaev \cite{RT} gave a rigorous
construction of $3$-manifold invariants associated to $\text{SU}(2)$. These invariants are called the \textit{Witten-Reshetikhin-Turaev} (WRT) invariants and are often denoted by $\tau_{\zeta}(M)$, where $\zeta$ is a $k$-th root of unity (here $k \in \mathbb{N}$ is called the {\em level}).

The concept of unified WRT invariants, introduced by \cite{H}, considers WRT invariants at all $k \in \mathbb{N}$.
For integral homology spheres, Habiro constructed invariants taking values in a completion $\widehat{\mathbb{Z}}[q] := \varprojlim \mathbb{Z}[q]/((q;q)_n)$, where, for $n\in\N_0\cup\{\infty\}$, $(a;q)_n:=\prod_{j=0}^{n-1} (1-aq^j)$, such that evaluation  at each root of unity $\zeta$ gives the WRT invariants. In a special case, invariants of this type (``analytic'' at roots of unity) appeared previously in work of Lawrence-Zagier \cite{LZ} on the Poincar\'e homology sphere $\Sigma(2,3,5)$ and work of Zagier \cite{Za1} on Vassiliev invariants. The latter paper was the first appearance of Kontsevich's intriguing series $F(q):=\sum_{n \geq 0} (q;q)_n$.  Quite remarkably, Zagier also constructed functions in the upper and lower half-plane that asymptotically agree with $f(\tau):=F(q)$ ($q \!:=\!e^{2\pi i \tau}, \tau \in \mathbb{H}$) at all rational numbers. This function constitutes an example of a quantum modular form. This notion was formalized by Zagier in \cite{Za2}, where he defined a {\it quantum modular form} to be a complex-valued function defined on the rational numbers that is essentially modular (up to a correction factor that has ``nice'' analytic properties).  Further examples of unified invariants (and quantum modular forms) for knots/links and 3-manifolds were considered by Hikami and other authors \cite{Hi1,Hi2,Hi3,Hikami,Hikami2}. There are other important aspects of quantum modular forms including  Maass forms, Eichler integrals, combinatorial generating functions, and meromorphic Jacobi forms. In addition, quantum modular forms recently appeared in representation theory of infinite-dimensional Lie algebras and vertex algebras in the form of characters of irreducible modules \cite{BM,CM}. Another interesting direction concerns $q$-series identities for expressions coming from unified WRT invariants \cite{BHL,Hikami2}.

In a very interesting recent paper \cite{GPPV}, Gukov, Pei, Putrov, and Vafa introduced a new approach to WRT invariants of 3-manifolds that was motivated by certain dualities in physics.
In this paper, among many other things, the authors defined
quantum invariants of various families of $3$-manifolds $\b{M}_3$, including plumbed 3-manifolds. A \textit{plumbed $3$-manifold}
$\b{M}_3({G})$ is associated to a labeled graph ${G}$, so that $\b{M}_3(G)$ is obtained by a Dehn surgery on the corresponding link of unknots. If the linking matrix
$M$ is positive definite\footnote{In \cite{GPPV}, $M$ is negative definite, which we have accounted for by replacing it by $-M$ when referring to their work.}, the authors in \cite{GPPV} defined a $q$-series denoted by $\widehat{Z}_{{\b{a}}}(q)$ with integral coefficients, and argued that the limiting values of $Z(q)$ at roots of unity are expected to capture the WRT invariants discussed above.
The key novelty of the approach is that the definition of $\widehat{Z}_{{\b{a}}}(q)$ is based on a straightforward contour integration, and the invariants are defined as functions in the upper half-plane rather than at roots of unity. In \cite{GPPV}, several examples of $\widehat{Z}_{{\b{a}}}(q)$ series  were computed in terms of unary false theta functions, which are known to be quantum modular forms (see for instance \cite{BM}).

We note that functions closely related to $\widehat{Z}_{{\b{a}}}(q)$ already appeared in the literature on the so-called higher rank singlet $W$-algebras denoted by $W(p)^0_{Q}$ \cite{BM2,CM}. In \cite{FG} a more
direct link between quantum invariants from \cite{GPPV} and vertex algebras was given.

\subsection{Statement of the results}

Based on several examples calculated in \cite{GPPV},  Gukov \cite{Gukov} conjectured a striking general characterization of the analytic properties of $\widehat{Z}_{{\b{a}}}(q)$.
\begin{conjecture}
\label{C:Gukov}
For any tree and labeling such that
$M$ is positive definite, $\widehat{Z}_{{\b{a}}}(q)$ is a quantum modular form.
\end{conjecture}

This paper is a first contribution towards the resolution of Gukov's conjecture.
We present a detailed analysis of  $\widehat{Z}_{0}(q)$ for unimodular graphs and manifolds coming from $n$-leg star graphs, which are also sometimes known as $n$-spiders, where $n$ denotes the number of leaves or legs. More precisely, we first introduce a closely related integral, denoted by $Z(q)$, for which we show that it agrees with $\widehat{Z}_{0}(q^2)$ for all unimodular positive definite plumbing matrices.
Our main result verifies the validity of the conjecture for all $3$-leg star graphs with unimodular plumbing matrices, and also provides strong evidence that quantum modularity is most likely true for general star graphs if we allow the quantum set to be a proper subset of $\mathbb{Q}$.
\begin{theorem}
\label{T:main}
	\begin{enumerate}[leftmargin=*, label=\textnormal{(\arabic*)},align=left]
		\item \label{T:main:3leg} For any $3$-leg star graph, there is some $c_M \in \Q$ such that $q^{c_M}Z(q)$ is a quantum modular form.
		\item \label{T:main:nleg} For any $n$-leg star graph, we can write $Z(q) = F_1(q) + \cdots F_r(q)$, where there are $c_{M,j} \in \Q$ such that each $q^{c_{M,j}} F_j(q)$ is a quantum modular form.
			\end{enumerate}
\end{theorem}

As a corollary, we obtain that for any positive definite unimodular $3$-leg star plumbing graph, there is some $c_M \in \mathbb{Q}$ such that $q^{c_M} \widehat{Z}_{0}(q)$ is a quantum modular form of weight $\frac12$ with quantum set $\mathbb{Q}$.

\begin{remarks}
(1) To clarify, the main difference between the two statements in Theorem \ref{T:main} is that in the case of $n$-leg star graphs, the quantum sets of the summands do not necessarily coincide. For more precise statements, see Theorems \ref{T:Zqquantum} and \ref{n-star} below.\\
(2) In a paper \cite{CCFGH} that appeared as a preliminary version of this paper was ready, Cheng, Chun, Ferrari, Gukov, and Harrison independently calculated  $\widehat{Z}_{\b{a}}(q)$ for a large number of additional examples of $3$-star graphs, as well as an example of a 4-star graph (see Section 8 of \cite{CCFGH}).
\end{remarks}

The paper is organized as follows. In Section 2, we introduce the function $Z(q)$ and prove that it agrees with $\widehat{Z}_{0}(q)$ from  \cite{GPPV} for any unimodular  plumbing matrix. In Section 3 we present elementary facts on
rational functions and define the quantum modular forms in this paper. In Section 4 we prove our first main result, Theorem \ref{T:Zqquantum}, on quantum modularity of $3$-leg star graphs.
Section 5 contains explicit computations of $Z(q)$ for all $D$ and $E$ type Dynkin diagrams. In Section 6 we prove a version of quantum modularity for
all star graphs.  We also present an example illustrating that the quantum set of $Z(q)$ can be smaller than $\mathbb{Q}$. We end in Section 7 with several remarks and directions for future work.

\section*{Acknowledgments}
 We thank S. Gukov for correspondence regarding a preliminary version of this paper, H. Hikami for bringing \cite{Hi1,Hi2,Hi3} to our attention, and the referees for their many helpful comments.

\section{Definitions and notation}

\subsection{The quantum invariant}

Consider a tree $G=(V,E)$ \footnote{Note that in \cite{GPPV} $G$ was used to denote the gauge group (which is not discussed in this paper).} with $N$ vertices numbered $1, 2, \ldots, N$. For such a $G$, we choose a symmetric integral matrix $M=(m_{jk})_{1 \leq j,k \leq N}$ such that $m_{jk}=-1$ if vertex $j$ is connected to vertex $k$ and zero otherwise. The diagonal entries $m_{j j} \in \mathbb{Z}$, $1 \leq j \leq N$ are known as {\em framing coefficients}, which for our purposes, may be freely chosen subject to the restriction that $M$ is positive definite (which ensures that the integrals below actually define Laurent series). We typically label the vertices of $G$ by these coefficients.

The first homology group of $\b{M}_3:=\b{M}_3(G)$ (plumbed $3$-manifold
constructed from $G$) is
$$H_1(\b{M}_3,\mathbb{Z}) \cong {\rm coker}(M) =\mathbb{Z}^N/M\mathbb{Z}^N.$$
 If $M$ is invertible then this group is finite and if $M \in \text{SL}_N(\mathbb{Z})$   (e.g. $G=E_8$),
then $H_1(\b{M}_3,\mathbb{Z})=0$.

To each edge joining vertices $j \neq k$ in $G$, we associate a rational function
$$
f(w_j,w_k):=\frac{1}{\big(w_j-w_j^{-1}\big)\big(w_k-w_k^{-1}\big)}
$$
and to each vertex $w_j$ a Laurent polynomial
$$
g(w_j):= \big(w_j - w_j^{-1}\big)^2.
$$

In \cite{GPPV}, a quantum invariant for plumbing graphs was introduced. In particular, given $G$ and $M$, set (throughout we use the vector notation ${\b w}:=(w_1,...,w_N)^T$)
\begin{equation}
\label{E:widehatZ}
\widehat{Z}_{\b a} (q):= \frac{q^{\frac{-3N+{\rm tr}(M)}{4}}}{(2\pi i)^N} \text{PV} \int_{|w_j|=1} \prod_{j=1}^N g(w_j) \prod_{(k,\ell) \in E} f(w_k,w_\ell) \Theta_{-M,\b a}(q;{{\b w}}) \frac{dw_j}{w_j},
\end{equation}
where  ${\rm PV}$ means the Cauchy principal value. Here
$\int_{|w_j|=1}$ indicates the integration $\int_{|w_1|=1}$ $\cdots \int_{|w_N|=1}$, and the theta function is defined by
\begin{equation*}
\Theta_{-M,\b a}(q;{\b w})=\sum_{{\bb \ell}  \in 2 M \mathbb{Z}^N + \b a} q^{\frac14\b{\ell}^T M^{-1} \b{\ell}} \prod_{j=1}^N w_j^{\ell_j}, \ \ \b a \in 2{\rm coker}(M)+ \bb \delta;
\end{equation*}
the shift is defined by $\bb \delta:=(\delta_j)$ such that $\delta_j \equiv {\rm deg}(v_j) \pmod 2$, where $\delta_j$ denotes the degree (valency) of the $j$-th node. Furthermore, the simplest shift is denoted in \cite{GPPV} by $\widehat{Z}_{0}(q) := \widehat{Z}_{\bb{\delta}}(q)$.

\subsection{The $Z$-integral}

We now introduce a closely related integral that is somewhat easier to work with analytically, and is equivalent to $\widehat{Z}_0(q)$ in an important special case, as we see below. For positive definite $M$ as above, we define
\begin{equation}
\label{Our:ZqDef}
Z(q):= \frac{q^{\frac{-3N+{\rm tr}(M)}{2}}}{(2\pi i)^N} \text{PV} \int_{|w_j|=1} \prod_{j=1}^N g(w_j) \prod_{(k,\ell) \in E} f(w_k,w_\ell) \Theta_{M}(q;{{\b w}}) \frac{dw_j}{w_j},
\end{equation}
where
\begin{equation} \label{E:Theta2}
\Theta_{M}(q;{\b w}):=\sum_{{\b{m}}  \in  M \mathbb{Z}^N } q^{\frac12\b{m}^T M^{-1} \b{m}} \prod_{j=1}^N w_j^{m_j}.
\end{equation}
Note that the only differences between \eqref{E:widehatZ} and \eqref{Our:ZqDef} are in the normalization of the $q$-powers and the condition in the summation in the theta functions.

\section{Preliminaries}
\label{S:Prelim}

\subsection{Basic facts for rational functions}
\label{S:Prelim:ratl}

In order to calculate $Z(q)$, we need to compute principal value integrals for various series and rational functions. We first observe that for $m\in\Z$,
\begin{equation}
\label{E:CTw-w^-1}
\begin{aligned}
\frac{1}{2 \pi i}  {\rm PV} \int_{|w|=1}\left(w-w^{-1}\right)w^m\frac{dw}{w}
&= \frac{1}{2\pi i} \int_{|w|=1}\left(w-w^{-1}\right)w^m\frac{dw}{w} \\
& = \text{CT}_{w}\left(\left(w-w^{-1}\right)w^{m}\right)
=\delta_{m=-1}-\delta_{m=1},
\end{aligned}
\end{equation}
where $\operatorname{CT}_w(h(w))$ denotes the constant term of the meromorphic function $f$ around $w=0$, and for a predicate $P$, we use the indicator notation
\begin{equation*}
\delta_P := \begin{cases} 1 \quad &\text{if } P \text{ is true}, \\
0 & \text{if } P \text{ is false}. \end{cases}
\end{equation*}

We next give a simple test that we use throughout the paper in order to reduce principal value integrals to constant term evaluations.
\begin{lemma}
	\label{L:removable}
	Suppose that $h(w)=\sum_{\ell \in \mathbb{Z}} a_\ell w^\ell$ satisfies
	$h(w^{-1})=- h(w)w^{2m}$ for some $m\in \Z$.
	Then
	$$\frac{1}{2 \pi i} {\rm PV} \int_{|w|=1} \frac{h(w) dw}{w(w-w^{-1})}={\rm CT}_w \left(\frac{h(w)}{w-w^{-1}}\right).$$
\end{lemma}
\begin{proof}
The only poles in the integrand are at $w = 0, \pm 1$, and the lemma statement amounts to the claim that the residues at $w = \pm 1$ make no contribution. This is verified by calculating
\begin{align*}
{\rm Res}_{w=\pm 1} &\left( \frac{h(w)}{w\left(w-w^{-1}\right)} \right)
= \frac12 {\rm Res}_{w= \pm 1} \left(\frac{h(w) - h(w^{-1}) w^{-2m}}{w^2 - 1} \right) \\
& = \frac12 \lim_{w \to \pm 1} (w \mp 1) \frac{h(w) - h\left(w^{-1}\right) w^{-2m}}{w^2 - 1}
= \pm \frac14 \lim_{w \to 1} \left(h(w) - h\left(w^{-1}\right) w^{-2m}\right) = 0.
\end{align*}
\vskip-2em
\end{proof}
\vskip1em
We can then make use of the identity
\begin{equation}
\label{E:CT1/w-w^-1}
\text{CT}_{w}\left(\frac{w^m}{w-w^{-1}}\right)=\begin{cases}
-1\quad&\text{if }m\leq -1 \text{ is odd},\\
0&\text{otherwise.}
\end{cases}
\end{equation}
Furthermore, \eqref{E:CT1/w-w^-1} also implies that
\begin{equation}
\label{E:CT=sgno}
\text{CT}_{w}\left(\frac{w^m - w^{-m}}{w-w^{-1}}\right) = \begin{cases}
\sgn(m)&\quad \textnormal{ if $m$ is odd,}\\
0&\quad \textnormal{ otherwise.}
\end{cases}
\end{equation}
In fact, there is a more general identity for $m \in \Z,  \ell \in \mathbb{N}_0$, namely
\begin{equation}
\label{E:CTmn}
{\rm CT}_w \left(\frac{w^m - w^{-m}}{w^\ell \left(w - w^{-1}\right)}\right)
= \delta_{|m| \geq \ell+1} \ \delta_{m \equiv \ell+1 \pmod{2}} \ \sgn(m).
\end{equation}

The following result helps us reduce general residue calculations to the above cases, which are straightforward as they only involve simple poles.
\begin{proposition}
\label{P:Resa}
Suppose that $h$ is a meromorphic function.
For any $\ell\in\N$ and $a \in \C$, we have
\begin{align*}
{\rm Res}_{w = a} \left(\frac{h(w)}{w (w-w^{-1})^{\ell+1}}\right)
= \frac{1}{2\ell} {\rm Res}_{w = a} \left(\frac{1}{w^\ell (w-w^{-1})^\ell} \frac{d}{d w} \left(h(w) w^{\ell-1}\right)\right).
\end{align*}
\end{proposition}
\begin{proof}
The claim follows from the simple fact that for meromorphic functions $h_1$ and $h_2$, we have
\begin{align*}
	{\rm Res}_{w=a} \left(h_1'(w) h_2(w)\right) = - {\rm Res}_{w=a} \left(h_1(w)h_2'(w)\right).
\end{align*}
The specific shape of the statement then follows from the derivative evaluation
\begin{align*}
w^{\ell-1} \frac{d}{d w} \left(\frac{1}{w^\ell(w-w^{-1})^\ell}\right)=-\frac{2\ell}{w(w-w^{-1})^{\ell+1}}.
\end{align*}
\vskip-2.5em
\end{proof}
\vskip2em
As an immediate application, we see that this aids in the calculation of the constant term for more complicated rational functions. First, we recall the (rising) Pochhammer symbol, $(x)_n := \prod_{j=0}^{n-1} (x+j)$ for $ n\in \N_0$. We also need the simple symmetry relation
\begin{equation}
\label{E:Poch}
(-x-n+1)_n=(-1)^n (x)_n.
\end{equation}

\begin{corollary}
\label{C:CTn}
For $a \in \C$, $m \in \Z$ and $\ell\in \N$, we have
\begin{equation*}
{\rm Res}_{w=a} \left(\frac{w^m}{w\left(w - w^{-1}\right)^\ell}\right)
= \frac{\left(\frac{m-\ell}{2}+1\right)_{\ell-1}}{(\ell-1)!} {\rm Res}_{w=a} \left(\frac{w^{m-\ell+1}}{w(w-w^{-1})}\right).
\end{equation*}
\end{corollary}
\begin{proof}
Proposition \ref{P:Resa} with $f(w)=w^m$ and $\ell\mapsto \ell-1$ implies that
\begin{align}\label{reduceres}
{\rm Res}_{w = a} \left(\frac{w^m}{w\left(w - w^{-1}\right)^\ell}\right)
= \frac{m+\ell-2}{2(\ell-1)} \, {\rm Res}_{w=a} \left(\frac{w^{m-1}}{w(w-w^{-1})^{\ell-1}}\right).
\end{align}
Writing $\frac{m + \ell - 2}{2} = \frac{m-\ell}{2} + \ell-1$, and then iterating \eqref{reduceres} $\ell-2$ more times gives the result.
\end{proof}

\subsection{Unimodular matrices}
Here we restrict to the case of unimodular matrices.

\begin{proposition} Let $M$ be a unimodular positive matrix, then $\widehat{Z}_{\bb  \delta}(q^2)=Z(q)$.
\end{proposition}
\begin{proof}  Recall that in the definition of $\widehat{Z}_{\bb \delta}(q^2)$ the summation in the theta function is over
$2 \mathbb{Z}^N+ \bb{\delta}$, whereas in the definition of $Z(q)$ we are summing over the full lattice $\mathbb{Z}^N$.

If vertex $j$ is a leaf (of degree one), then $\delta_j=1$, so the summation in $\widehat{Z}_{\bb \delta}$ is over odd integers. But according to \eqref{E:CT1/w-w^-1}, the only non-zero contributions in $Z(q)$ also come from odd powers.

For vertices of degree two we have $\delta_j=0$, and thus there is no rational function in $w_j$, so the only contribution comes from the constant term - the zeroth (which is even) power of the variable.

For vertices of odd degree $r \geq 3$, $\delta_j=1$, and from Corollary \ref{C:CTn} we obtain for $m$ even
\begin{equation*}
\frac{1}{2\pi i}\text{PV}\int_{|w_j|=1}\frac{w_j^m}{\left( w_j-w_j^{-1}\right)^{r-2}}\frac{dw_j}{w_j}=0.
\end{equation*}

Finally, for vertices of even degree $r \geq 4$, $\delta_j=0$, similarly we have, for $m$ odd
\begin{equation*}
\frac{1}{2\pi i}\text{PV}\int_{|w_j|=1}\frac{w_j^m}{\left( w_j-w_j^{-1}\right) ^{r-2}}\frac{dw_j}{w_j}=0.
\end{equation*}
The statement of the proposition follows.
\end{proof}

\subsection{An example with a non-unimodular matrix}

\medskip

In this section we calculate an example of $\widehat{Z}_{0}(q)$ for a non-unimodular matrix. Although this computation can be carried out with minimal machinery (and was previously computed as the $k=1$ case of $\widehat{Z}_{0}(q)$ in \cite[(6.81)]{CCFGH}), it is easiest to carry out here if we make use of some of our general results. We consider the Dynkin diagram for $D_4$ (see Figure \ref{F:DN+2}), and its Cartan matrix
$$
M = \left(\begin{smallmatrix}
2 & -1 & 0  & 0 \\
-1 & 2 & -1 & -2 \\
0& -1 & 2 & 0\\
0& -1 & 0& 2
\end{smallmatrix}\right).
$$
Noting that
\begin{align*}
2M\Z^4 + (1,1,1,1)= M \left(2 \Z + 1\right)^4,
\end{align*}
we can therefore write $\bb{\ell} = M(2\bb{n}+\bb{1})$, where $\bb{1}:=(1,1,1,1)$. We then calculate
\begin{align*}
 \widehat{Z}_{\b 1}(q)&= \frac{q^{-1}}{(2\pi i)^4} {\rm PV} \int_{|w_j| = 1} \frac{\prod_{k\in\{1,3,4\}}\left(w_k-w_k^{-1}\right)}{w_2 - w_2^{-1}}\sum_{\bb\ell\in M(2\Z+1)^4} q^{\frac14 \bb{\ell}^T M^{-1} \bb{\ell}} \prod_{k=1}^4 w_k^{\ell_k} \prod_{1\leq j\leq 4} \frac{dw_j}{w_j} \\
& = \frac{q^{-1}}{(2\pi i)^4} {\rm PV} \int\limits_{|w_j| = 1} \frac{\prod_{k\in\{1,3,4\}}\left(w_k-w_k^{-1}\right)}{w_2 - w_2^{-1}}
\sum_{\bb{n}\in\Z^4} q^{\frac14(2\bb{n}+\bb{1})^TM(2\bb{n}+\bb{1})} \\
& \hspace{3.4cm} \times w_1^{4n_1-2n_2+1} w_2^{4n_2-2n_1-2n_3-2n_4-1}w_3^{4n_3-2n_2+1} w_4^{4n_4-2n_2+1} \prod_{1\leq j\leq4} \frac{dw_j}{w_j}.
\end{align*}
Using \eqref{E:CTw-w^-1}, the integrals on $w_1, w_3, w_4$ force the exponents to be $\pm 1$. Furthermore, if the exponent on $w_1$ is 1, then $n_2 = 2n_1$. In particular, $n_2$ is even, so the exponent of $w_3$ is then $4n_3 - 2n_2 + 1 \equiv 1 \pmod{4}$, and thus it cannot be $w_3^{-1}$; similarly for $w_4$. This shows that the only possibilities are $w_1 w_3 w_4$ or $w_1^{-1} w_3^{-1} w_4^{-1}$. In the first case the remaining sum is parametrized by $n_2 = 2n$, and $n_1 = n_3 = n_4 = n$, and in the second case by $n_2 = 2n+1$, again with $n_1 = n_3 = n_4 = n$. A short residue calculation (as in Section \ref{S:Prelim:ratl}) then gives
\begin{align*}
\widehat{Z}_{\b 1}(q) &= \frac{q^{-1}}{2}  \Bigg(-\sum_{n \in \Z} \sgn(2n-1) q^{\frac12 \left(3(2n+1)^2 + (4n+1)^2 - 3(4n+1)(2n+1)\right)} \\
& \hspace{4.4cm}  + \sum_{n \in \Z}\sgn(2n+3) q^{\frac12 \left(3(2n+1)^2 + (4n+3)^2 - 3(4n+3)(2n+1)\right)}\Bigg) \\ \notag
& =\frac{q^{-1}}{2} \left(-\sum_{n \in \Z}\sgn^\ast(n) q^{\frac12 \left(4n^2 + 2n + 1\right)} +\sum_{n \in \Z}\sgn^\ast(n+1) q^{\frac12 \left(4n^2 + 6n + 3\right)} \right) \\ \notag
&  =  \sum_{n \in \Z} \sgn^*(n) q^{\frac12 \left(4n^2 -2n - 1\right)}.
\end{align*}
For the second equality we change $n\mapsto -n$ in the first sum and $n\mapsto n-1$ in the second sum, and we also use the notation $\sgn^*(x):=\sgn(x)$ for $x\in\mathbb{R}\setminus\{0\}$, and $\sgn^*(0):=1$. As a point of comparison, this is different from $Z_{D_4}(q)$ in Proposition \ref{P:DN+2}.

\subsection{Quantum modular forms}
\label{S:Prelim:Quantum}

Define the following false theta functions ($j,p \in \Z$)
\begin{equation}
\label{E:Fjp}
	F_{j,p}(\tau):=\sum_{m\in\mathbb Z}\\\sgn^*(m)q^{\left(m+\frac{j}{2p}\right)^2}.
\end{equation}
This definition may easily be extended to rational $j$ and $p$; indeed, if $j = \frac{r}{s}$ and $p = \frac{h}{k}$, then $F_{j,p} = F_{rk, hs}$. We also note the identities
\begin{align}\label{rel1}
F_{-1,p}(2p\tau) &= 2q^{\frac{1}{2p}} -\sum_{n\equiv 1\pmod{2p}} \sgn(n) q^{\frac{n^2}{2p}} ,\\ \label{rel2}
F_{a,p}(2p\tau) &=  -\sum_{n\equiv 2p-a \pmod{2p}} \sgn(n) q^{\frac{n^2}{2p}}\qquad\quad (0<a<2p).
\end{align}
The second of these follows by substituting $n = 2pm + a$ in \eqref{E:Fjp}, whereas the first requires setting $n = 1 - 2pm$, and then separating the $n=1$ term.

In \cite{Za2} Zagier defined a {\em quantum modular form} to be a function $f:\mathcal{Q}\rightarrow\C$ $(\mathcal{Q}\subseteq \Q)$, such that the {\it obstruction to modularity} ($M = (\begin{smallmatrix}
a&b\\c&d
\end{smallmatrix})\in\SL_2(\Z)$)
\begin{equation}\label{errormod}
f(\tau)-(c\tau+d)^{-k} f(M\tau)
\end{equation}
has ``nice'' analytic properties. The definition is intentionally vague to include many examples; in this paper we require \eqref{errormod} to be real-analytic. This assumption is particularly useful because it guarantees that quantum modularity is preserved under differentiation, which is needed in the proof of Theorem \ref{T:main} \ref{T:main:nleg} (see Section \ref{genstar}).

It was shown in Section 4 of \cite{BM} that the $F_{j,p}$ are quantum modular forms. In order to describe this precisely, we first recall that
$$\Gamma_1(n):= \biggl\{ \begin{pmatrix} a & b \\ c &d \end{pmatrix}  \in \SL_2(\mathbb{Z}) : a,d \equiv 1 \ ({\rm mod} \ n), c \equiv 0 \ ({\rm mod} \ n) \biggr\}.$$
\begin{proposition}
\label{quantum-12}
The functions $F_{j,p}$ are quantum modular forms of weight $\frac12$ on $\Gamma_1(4p)$ (with explicit multiplier systems) and quantum set $\mathbb Q$.
\end{proposition}

The asymptotic expansion of $F_{j,p}(it)$ as $t\to 0^+$ is given by Corollary 4.5 of \cite{BM}, as
\begin{equation}
\label{E:Fasymp}
F_{j,p}(it) \sim -2 \sum_{r\geq 0 } \frac{B_{2r+1}\left(\frac{j}{2p}\right)}{2r+1}\frac{(-2\pi t)^r}{r!},
\end{equation}
where $B_\ell(x)$ denotes the $\ell$-th Bernoulli polynomial. In particular, $B_1(x)=x-\frac12$ and $B_3(x)=x^3-\frac{3x^2}{2}+\frac{x}{2}.$

We also require certain weight $\frac32$ quantum modular forms. Set
\begin{equation*}
G_{j,p}(\tau) := \sum_{m\in\Z} \left|m+\tfrac{j}{2p}\right| q^{\left(m+\frac{j}{2p}\right)^2}.
\end{equation*}
The quantum modularity properties of these functions were given in Section 6.1 of \cite{BCR} (up to finitely many terms $G_{j,p}(\tau)$ is $\Theta_{\frac32}^+ (2p,j;\tau)$ from \cite{BCR}). To state this result, define
$$\widehat{\mathcal{Q}}_{2p,j} := \begin{cases}
\left\{ \frac{h}{k} \in \mathbb{Q}: {\rm gcd}(h, k)=1,  \  p \ | \ k, {\rm ord}_2(k)={\rm ord}_2(p) \right\} &   {\rm if} \ p\nmid j, \\[2pt]
\left\{ \frac{h}{k} \in \mathbb{Q}: {\rm gcd}(h, k)=1, \   {\rm ord}_2(k)>{\rm ord}_2(p)+1  \right\} &   {\rm if} \ j \equiv p  \pmod{2p}, \\[2pt]
\left\{ \frac{h}{k} \in \mathbb{Q}: {\rm gcd}(h, k)=1, \ {\rm ord}_2(k)={\rm ord}_2(p)+1 \right\} &   {\rm if} \ j \equiv 0 \pmod{2p},
\end{cases}   $$
where for $k \in \mathbb{Z}$, ${\rm ord}_2(k) = \nu$ if $2^\nu$ is the largest power of $2$ dividing $k$.

\begin{proposition} \label{quantum-32}
The functions $\tau\mapsto G_{j,p}(p\tau)$ are quantum modular forms of weight $\frac32$ on $\Gamma_1(4p)$ (and explicit multiplier systems) with quantum set $\widehat{\mathcal{Q}}_{2p,j}$.

\end{proposition}
\begin{remark}
One inconvenience in working with weight $\frac32$ quantum modular forms $G_{j,p}$ is that for different values of $j$ and fixed $N$ the quantum sets can be disjoint.
\end{remark}

The next result follows directly from the definitions.
\begin{proposition} \label{quantum_set} Let $\tau\mapsto H(\tau)$ be a quantum modular form of weight $\frac12$ or $\frac32$ with respect to a subgroup of $\SL_{2}(\Z)$ with a quantum set $S \subset \mathbb{Q}$, and let $a \in \mathbb{Q}^+$. Then $\tau\mapsto H(a \tau)$ is also quantum modular  on the set  $\frac1a \cdot S$ (with respect to a subgroup of $\SL_2(\Z)$).
\end{proposition}

\section{$3$-leg star graphs and the proof of Theorem \ref{T:main} \ref{T:main:3leg}}

An {\it $(\ell-1)$-leg star graph} consists of $\ell-1$ legs joined to a central vertex. We enumerate the nodes as indicated in Figure \ref{F:ellstar}, with the vertex of degree $\ell-1$ labeled by $\ell$, and the external nodes (of degree 1) by $1,...,\ell-1$. Furthermore, an {\it $(\ell-1)$-star graph} is such a graph in which none of the legs have any interior nodes; in other words, this is a tree with $\ell$ vertices, where one central vertex is connected to $(\ell-1)$ leaves.

\begin{figure}[h]
\begin{center}

	\begin{tikzpicture}[thick, scale=1.2]
	\node[shape=circle,fill=black, scale = 0.4] (1) at (-0.3,0) { };
	\node[shape=circle,fill=black, scale = 0.4] (2) at (0.5,1) { };
	 \node[shape=circle,fill=white, scale = 0.1] (3) at (1.5,1) { };
	\node[shape=circle,fill=black, scale = 0.4] (4) at (2.5,1) { };
	   \node[shape=circle,fill=black, scale = 0.4] (5) at (3.5,1) { };

\node[shape=circle,fill=black, scale = 0.4] (6) at (0.5,0.5) { };
\node[shape=circle,fill=white, scale = 0.1] (13) at (1.5,0.5) { };
\node[shape=circle,fill=black, scale = 0.4] (14) at (2.5,0.5) { };
\node[shape=circle,fill=black, scale = 0.4] (15) at (3.5,0.5) { };

\node[shape=circle,fill=white, scale = 0.1] (7) at (0.5,0) { };
\node[shape=circle,fill=white, scale = 0.1] (8) at (0.5,-0.5) { };

	\node[shape=circle,fill=black, scale = 0.4] (9) at (0.5,-1) { };
	 \node[shape=circle,fill=white, scale = 0.1] (10) at (1.5,-1) { };
	\node[shape=circle,fill=black, scale = 0.4] (11) at (2.5,-1) { };
	   \node[shape=circle,fill=black, scale = 0.4] (12) at (3.5,-1) { };

	\node[draw=none] (A1) at (-0.4,-0.3) {$a_{\ell}$};
	\node[draw=none] (A2) at (3.7,1.3) {$a_{1}$};
	\node[draw=none] (A3) at (3.7,0.7) {$a_{2}$};
	\node[draw=none] (A4) at (3.8,-0.7) {$a_{\ell-1}$};

	\path [-](1) edge node[left] {} (2);
	
	\path [dotted] (2) edge node[left] {} (3);
	\path [dotted] (3) edge node[left] {} (4);
	\path [-] (4) edge node[left] {} (5);
	
	\path [-] (1) edge node[left] {} (6);
	\path [dotted] (1) edge node[left] {} (7);
	\path [dotted] (1) edge node[left] {} (8);

\path [-] (1) edge node[left] {} (9);
\path [-] (1) edge node[left] {} (6);
\path [dotted] (6) edge node[left] {} (13);
\path [dotted] (13) edge node[left] {} (14);
\path [-] (14) edge node[left] {} (15);

\path [dotted] (9) edge node[left] {} (10);
\path [dotted] (10) edge node[left] {} (11);
\path [-] (11) edge node[left] {} (12);

	\end{tikzpicture}
		
\end{center}
\caption{An $(\ell-1)$-leg star graph.} \label{F:ellstar}
\end{figure}
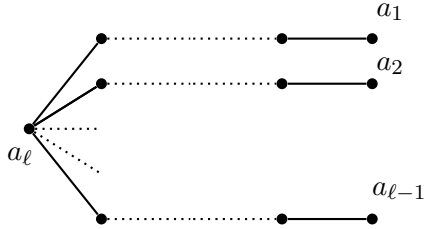

In this section we consider $3$-leg star graphs, and prove Theorem \ref{T:main} \ref{T:main:3leg}.

\subsection{Singularities}
\label{S:Sing}

We consider the singularities of the integrand in \eqref{Our:ZqDef} in the case of $3$-leg star graphs, and show that the principal value integral is not needed.
We have
\begin{align*}
Z(q)=\frac{q^{\frac{-3N + \sum_{\nu = 1}^N a_\nu}{2}}}{(2\pi i)^N}\text{PV}\int_{|w_4|=1}\frac{1}{w_4-w_4^{-1}}h(w_4)\frac{dw_4}{w_4},
\end{align*}
where
\begin{equation*}
h(w_4):=\int_{|w_k|=1}\Theta_M(\b w)\prod_{r=1}^{3}\left(w_r-w_r^{-1}\right)\prod_{\substack{k = 1 \\ k\neq 4}}^N\frac{dw_k}{w_k}.
\end{equation*}

\begin{proposition}
\label{P:Z3star}
For $3$-leg star graphs, we have
\begin{equation}
\label{E:Zq=CT}
Z(q) =q^{\frac{-3N + \sum_{\nu = 1}^N a_\nu}{2}} \textnormal{CT}_{\b w}\left(\frac{	 \Theta_M(\b{w})\prod_{r\in\{1,2,3\}}\left(w_r-w_r^{-1}\right)}{w_4-w_4^{-1}}\right).
\end{equation}
\end{proposition}
\begin{proof}
First, observe that changing $w\mapsto w^{-1}$ gives
\begin{align}
\label{E:winvert}
\int_{|w| = 1} w^n \frac{dw}{w} &= \int_{|w| = 1} w^{-n} \frac{dw}{w},\qquad
\int_{|w|=1}  \left(w-\frac{1}{w}\right)w^n \frac{dw}{w} =- \int_{|w|=1} \left(w-\frac{1}{w}\right)w^{-n} \frac{dw}{w}.
\end{align}
Furthermore, changing $\b n\mapsto -\b n$ in the summation of $\Theta_M(\b w)$  implies that
\[
\Theta_M(\b w)= \Theta_M\left(\b w^{-1}\right),
\]
where $\b w^{-1}:=(w_1^{-1},\ldots,w_N^{-1})$. Combined with \eqref{E:winvert} this then directly gives that
\begin{align*}
h\left(w_4^{-1}\right)& =- h(w_4).
\end{align*}
The conclusion follows from Lemma \ref{L:removable}.
\end{proof}

\subsection{Quantum modularity}
\label{S:Quantum}

Here we prove that $Z(q)$  is a quantum modular form for every  3-leg star  graph whose $M$ matrix is positive definite.

\begin{theorem}\label{T:Zqquantum}
	For any $3$-leg star graph, $Z(q)$ is a linear combination of quantum modular forms (up to $q$-powers) of weight $\frac12$ and quantum set $\mathbb{Q}$. In particular, the statement of Conjecture \ref{C:Gukov} is true for unimodular linking matrices.
\end{theorem}
We begin by considering the special case where $Z(q)$ is the $3$-star graph (which has just four vertices).
If $A = (a_{jk})_{1\leq j,k\leq 4}$ is a positive definite symmetric $4\times4$ matrix with rational entries, then we define, with $Q(\b{m}) := \frac12 \b{m}^T A {\b m}$,
\begin{equation}
\label{E:ZAq}
Z_A(q) := {\rm CT}_{\b w} \left( \frac{\prod_{r=1}^3\left(w_r-w_r^{-1}\right)}{w_4-w_4^{-1}} \sum_{{\b m} \in \Z^4 \cap A^{-1} \Z^4}  q^{Q(\b{m})} e^{2\pi i\b m^T\b z} \right).
\end{equation}

We next express $Z_A(q)$ as a linear combination of false theta functions for a large class of matrices. In particular, up to a rational $q$-power, each of these terms is a quantum modular form.

\begin{proposition} \label{P:quantum}  Suppose that $A \b{m} \in \Z^4$ for all $\b m\in \Z^4$ such that $m_1, m_2, m_3 \in \{\pm 1\}$, and $m_4$ is odd. Then we have, with the $b_j$ and $c_j$ defined in \eqref{defineb} and $d_j:=c_j-\frac{b_j^2}{2a_{44}}$,
\begin{equation}\label{Zfalse}
Z_A(q) = \sum_{j=1}^4 q^{d_j} F_{a_{44}-b_j,a_{44}}(2a_{44}\tau).
\end{equation}

As $t \to 0^+$, we have
\begin{align}\label{AsZA}
Z_A\left(e^{-2\pi t
	}\right) \sim -\frac{8 \pi}{a_{44}^2} \Big(a_{44} \left(a_{12} a_{34} + a_{13} a_{24} + a_{14} a_{23}\right) - 2 a_{14} a_{24} a_{34}\Big)t.
\end{align}
\end{proposition}

\begin{proof}

By assumption we have $A \b m\in\Z^4$ for any vector with $m_1, m_2, m_3 \in \{\pm1\}$, and $m_4$ odd. We may therefore evaluate the constant terms with respect to $w_1, w_2,$ and $w_3$ in \eqref{E:ZAq} using \eqref{E:CTw-w^-1}, which gives that
\begin{align}\label{E:ZADef}
\begin{split} Z_A(q)&=-{\rm CT}_{w} \left( \left(w-w^{-1}\right)^{-1}\sum_{m_1, m_2, m_3 \in \{\pm 1\} } m_1m_2m_3  \sum_{m_4 \in \mathbb{Z}}
	q^{Q(\b{m})}w^{m_4} \right)\\
&={\rm CT}_{w}  \left(\left(w-w^{-1}\right)^{-1} \sum_{j=1}^4 \sum_{\pm}\mp \sum_{m \in \mathbb{Z}} q^{  \frac{1}{2} a_{44}m^2 \pm b_j m+c_j } w^{m}\right),
\end{split}\end{align}
where for brevity we write  $\sum_{\pm} \mp \sum a(\pm)  :=- \sum a(+) + \sum a(-)$, and
\begin{align}\label{defineb}
b_1&:=\sum_{j=1}^3a_{j4},\qquad b_2:=a_{14}-a_{24}-a_{34},&\qquad b_3&:=-a_{14}-a_{24}+a_{34},\qquad b_4:=-a_{14}+a_{24}-a_{34},\\\notag
c_1&:=\frac12\sum_{j=1}^3 a_{jj} + \sum_{1\leq j<\ell\leq 3} a_{j\ell},&\qquad c_2&:=\frac12\sum_{j=1}^3 a_{jj}-a_{12}-a_{13}+a_{23},\\\notag
c_3&:=\frac12\sum_{j=1}^3 a_{jj}+a_{12}-a_{13}-a_{23},&\qquad c_4&:=\frac12\sum_{j=1}^3 a_{jj}-a_{12}+a_{13}-a_{23}.
\end{align}

Note that according to the definition of $Z_A(q)$, the sum in \eqref{E:ZADef} is over those $\b m$ which satisfy $A \b{m} \in \Z^4,$ and our assumption only guarantees that all odd $m_4$ are included. However, we  see below that any even values of $m_4$ vanish regardless, so we may write the sum over $m_4 \in \Z$.
Indeed, replacing $m \mapsto -m$ in the terms with a minus sign , and then applying \eqref{E:CT=sgno}, we obtain
\begin{align}
\label{E:CTw=odd}
\nonumber Z_A(q)&={\rm CT}_{w}  \left(\sum_{j=1}^4 \sum_{m \in \mathbb{Z}} q^{  \frac{1}{2} a_{44}m^2 - b_j m+c_j } \frac{w^m - w^{-m}}{w-w^{-1}}\right)\\
 &= \sum_{j=1}^4 \sum_{m \equiv 1\pmod{2}} {\rm sgn}(m) q^{  \frac{1}{2} a_{44}m^2 - b_j m+c_j }.
\end{align}

In general, we can write
\begin{align}\label{E:sgn=F}
\sum_{m\equiv 1\pmod{2}} \sgn(m) q^{\frac12 a m^2 - b m} =  q^{-\frac{b^2}{2a}} F_{a-b,a}(2a\tau).
\end{align}
Applying this to \eqref{E:CTw=odd}, we note that $2a_{44} \in \Z$ since by assumption $A (\pm 1, \pm 1, \pm 1, 1)^T\in\Z^4$, and adding these two relations gives the claim. Furthermore, since the fourth entry of $A (\varepsilon_1, \varepsilon_2, \varepsilon_3, 1)^T$ is $\varepsilon_1 a_{14} + \varepsilon_2 a_{24} + \varepsilon_3 a_{34} + a_{44}$ for any $\varepsilon_j \in \{\pm1\}$, we also conclude that $b_j -a_{44} \in \Z$ for all $j$. We therefore obtain \eqref{Zfalse}.

In order to prove \eqref{AsZA}, we plug in to \eqref{E:Fasymp}. A short calculation shows that $\sum_j b_j = 0$, which implies that the constant term in $Z_A(e^{-2 \pi t})$ vanishes. This leaves
\begin{equation*}
Z_A\left(e^{-2 \pi t}\right) \sim
-\frac{\pi}{3a_{44}^2} \left(6a_{44} \sum_{j=1}^{4} b_j c_j - 2 \sum_{j=1}^{4} b_j^3 - a_{44}^{2} \sum_{j=1}^{4}b_j \right) t,
\end{equation*}
which simplifies to the claimed expression (using the fact that $\sum_j b_j^3 = 24a_{14}a_{24}a_{34}$).
\end{proof}

Our next goal is to modify the proof above so that it applies to any 3-leg star graph. First we need to characterize the summation conditions that appear in the theta function \eqref{E:Theta2}.

\begin{lemma}
\label{help} Let $M \mathbb{Z}^N \subset \mathbb{Z}^N$ be an integral lattice and $D:={\rm det}(M)$.
Then there exist an additive subgroup $S$ of $(\Z / D\Z)^N$ such that $\b{m} \in M \Z^N$ if and only if $\b{m} \equiv \b{s} \pmod{D}$ for some $\b{s} \in S$.
\end{lemma}

\begin{proof}[Proof of Lemma \ref{help}] Clearly ${\b m} \in M \mathbb{Z}^N$ if and only if there exists $\b x\in\Z^N$ with ${\b m} = M {\b x}$. This is equivalent to $M^{-1} \b m= {\rm adj}(M){\rm det}(M)^{-1} \b m \in \mathbb{Z}^N$. Now consider $\varphi : \b{x} \mapsto {\rm adj}(M) \b{x},$ which is a group automorphism of $\left(\Z / D\Z\right)^N$. Thus the subgroup $S$ is simply the kernel  ${\rm ker}(\varphi).$
\end{proof}

We are now ready to prove Theorem \ref{T:Zqquantum}.
\begin{proof}[Proof of Theorem \ref{T:Zqquantum}]
The strategy is similar to the proof of Proposition \ref{P:quantum}. Beginning from \eqref{E:Theta2} and Proposition \ref{P:Z3star}, we first evaluate the constant terms in $w_5, \dots, w_N$, and find that
\begin{align}\notag
Z(q) & = q^{\frac{-3N + \sum_{\nu = 1}^N a_\nu}{2}} {\rm CT}_{\b{w}} \left(\frac{\prod_{r=1}^3 \left(w_r - w_r^{-1}\right)}{w_4 - w_4^{-1}} \sum_{\b{m} \in M \Z^N} q^{Q_2(\b{m})} e^{2\pi i\b m^T\b z}\right) \\\notag
& = q^{\frac{-3N + \sum_{\nu = 1}^N a_\nu}{2}} {\rm CT}_{w_1, w_2, w_3, w_4} \left(\frac{\prod_{r=1}^3 \left(w_r - w_r^{-1}\right)}{w_4 - w_4^{-1}} \hspace{-1em} \sum_{\substack{\b{m} = (m_1, m_2, m_3, m_4, 0, \cdots, 0)^T \in M \Z^N }}  \hspace{-1em}q^{Q_2(\b{m})}  e^{2\pi i\b m^T\b z} \right)\\
&=-q^{\frac{-3N + \sum_{\nu = 1}^N a_\nu}{2}} {\rm CT}_{w_4} \left(\sum_{\substack{m_1, m_2, m_3 \in \{\pm 1\} \\ \b{m} = (m_1, m_2, m_3, m_4, 0, \cdots, 0)^T \in M \Z^N }} m_1 m_2 m_3 \,  q^{Q_2(\b{m})} \frac{w_4^{m_4}}{w_4 - w_4^{-1}}\right), \label{E:CTw4}
\end{align}
where we use \eqref{E:CTw-w^-1} to evaluate the constant terms in $w_1, w_2$, and $w_3$.
Lemma \ref{help} implies that there exists an additive subgroup $S \subset (\Z / D \Z)^N$ such that $\b{m} \in M \Z^N$ if and only if $\b{m} \equiv \b{s} \pmod{D}$ for some $\b{s} \in S$. Now let $T \subset S$ be the subset of elements of the form $\b{s} = (\alpha, \beta, \gamma, g, 0, \cdots, 0)^T \in S$, where $\alpha, \beta, \gamma = \pm 1$ (interpreting $-1$ as a residue modulo $D$), and $g$ is a residue modulo $D$. The fact that $S$ is a subgroup then implies that $T = -T$ as well. Using this symmetry, we can therefore pair $\b{s}$ and $-\b{s}$, and write \eqref{E:CTw4} as
\begin{multline}
\label{E:CTw4=j}
-\frac12 q^{\frac{-3N + \sum_{\nu = 1}^N a_\nu}{2}} {\rm CT}_{w_4} \left(\sum_{\b{s} \in T} \; \sum_{\substack{ \b{m} = (\alpha, \beta, \gamma, m_4, 0, \cdots, 0)^T \\
m_4 \equiv g \pmod{D}}} \alpha \beta \gamma \, q^{Q_2(\b{m})} \frac{w_4^{m_4} - w_4^{-m_4}}{w_4 - w_4^{-1}}\right) \\
= -\frac12 q^{\frac{-3N + \sum_{\nu = 1}^N a_\nu}{2}} \sum_{\b{s} \in T} \; \sum_{\substack{ \b{m} = (\alpha, \beta, \gamma, m_4, 0, \cdots, 0)^T \\
m_4 \equiv g \pmod{D} \} \\ m_4 \equiv 1 \pmod{2}}} \alpha \beta \gamma \, \sgn(m_4) \, q^{Q_2(\b{m})},
\end{multline}
where we employ \eqref{E:CT=sgno}. Ignoring constant factors for fixed $\b{s}$, the inner sum has the form in \eqref{E:CTw4=j}
\begin{equation*}
\sum_{\substack{m \equiv g \pmod{D} \\ m \equiv 1 \pmod{2}}} \sgn(m) q^{\frac12 a m^2 + bm + c}
\end{equation*}
for certain $a,b,c\in\Q$.
The system
\begin{equation} \label{E:system}
m \equiv g \pmod{D}, \ \ \ \ m \equiv 1 \pmod{2}
\end{equation}
has a solution if and only if ${\rm gcd}(2,D) \mid (g-1)$, which splits into two cases depending on whether $D$ is odd, or whether $D$ is even and $g$ is odd.

If $D$ is odd, then we have a unique solution $h$ modulo $2D$ of \eqref{E:system}, so the sum turns into
\begin{align*}
	\sum_{m \equiv h\pmod{2D} } {\rm sgn}(m) q^{ \frac{1}{2} a m^2 + b  m+c }= q^{c - \frac{b^2}{2a}} F_{ha+b, a D}\left(2 a D^2 \tau\right)+r(q),
\end{align*}
where $r(q)$ is a finite sum of rational powers of $q$ (due to the shift in the {\rm sgn}-function the two series can have different signs for finitely many $m$).
As above, the $F_{j,p}$ are all quantum modular forms, and so is $r$.

If $D$ is even and $g$ is odd, then we have $m \equiv g \pmod{D}$ (and $D=2k$ is even),  so we get
	$$
\sum_{m \equiv g\pmod{2k}} {\rm sgn}(m) q^{ \frac{1}{2} a m^2 + b  m+c }
	= q^{c - \frac{b^2}{2a}} F_{ga+b, a k}\left(2a k^2 \tau\right)+r(q).
	$$

Combining all cases and recalling \eqref{E:CTw4}, we therefore conclude that $Z(q)$ is a finite sum of terms of the form $q^\varrho F(q)$, where $\varrho$ is rational and $F$ is a quantum modular form with quantum set $\Q$, which completes the proof.
\end{proof}

A closer analysis of the proof of Theorem \ref{T:Zqquantum} also provides a criterion for when the calculation of $Z(q)$ reduces to a 3-star graph, which can then often be computed using Proposition \ref{P:quantum}. Let $A$ be the restriction of ${\rm adj}(M)$ to rows and columns 1, 2, 3, and 4. Furthermore, define
\begin{align*}
\mathcal{V}& := \left\{ \b{m} = \left(m_1, m_2, m_3, m_4, 0, \cdots, 0\right)^T \; : \; m_1, m_2, m_3 \in \{\pm1\}, m_4 \in \Z \right\},\\
\Omega & := \left\{ \b{m} \in \mathcal{V} :  {\rm adj}(M) \b{m} \equiv 0 \pmod{D} \right\},\ \  \Omega_A  := \left\{ \b{m} \in \mathcal{V} : A \left(m_1, m_2, m_3, m_4\right)^T \equiv 0 \pmod{D} \right\}.
\end{align*}
By definition, $\Omega_A \subset \Omega$.
\begin{corollary}
\label{C:Z=ZA}
If $\Omega = \Omega_A$, then
\begin{equation*}
Z(q) = q^{\frac{-3N + \sum_{\nu = 1}^N a_\nu}{2}} Z_A(q).
\end{equation*}
Furthermore, if $M$ is unimodular, then Proposition \ref{P:quantum} always applies to $Z_A(q)$.
\end{corollary}

\begin{proof}
We begin by rewriting \eqref{E:CTw4} using the above notation, as well as the assumption that $\Omega = \Omega_A$, which gives
\begin{align*}
Z(q) &= -q^{\frac{-3N + \sum_{\nu = 1}^N a_\nu}{2}} {\rm CT}_{w_4} \left(\sum_{\b{m} \in \Omega} m_1 m_2 m_3 q^{Q_2(\b{m})} \frac{w_4^{m_4}}{w_4 - w_4^{-1}}\right) \\
&= -q^{\frac{-3N + \sum_{\nu = 1}^N a_\nu}{2}} {\rm CT}_{w_4} \left(\sum_{\b{m} \in \Omega_A} m_1 m_2 m_3 q^{Q_2(\b{m})} \frac{w_4^{m_4}}{w_4 - w_4^{-1}}\right).
\end{align*}
Evaluating and comparing to \eqref{E:ZAq}, one sees that this second constant term expression is simply $Z_A(q)$ (after evaluating the constant terms in $w_1, w_2,$ and $w_3$).

Finally, in the case that $M$ is unimodular, the congruences modulo $D$ are trivial, so $\Omega = \Omega_A = \mathcal V$. Proposition \ref{P:quantum} also clearly applies, since $A$ has integer entries. \qedhere

\end{proof}

\section{Explicit examples of $3$-leg star graphs}
\label{S:3starEx}
In this section we compute $Z(q)$ for all cases where $M$ is the Cartan matrix of a simple Lie group whose Dynkin diagram is a 3-leg star graph  (see \cite[pages 164--71]{OV}). In particular, throughout we write $Z_{\frak{g}}(q)$ to indicate $Z(q)$ in the case that the 3-leg star graph $G$ is the Dynkin diagram for the Lie algebra $\frak{g}$, with labeling matrix $M$ defined by $a_j = 2$.  Note that the vertex numbering in this section differs from our general constructions above. In all cases we write $Z_{\frak{g}}(q)$ as a quantum modular form, and determine its asymptotic behavior as $q \to 1^-$.
\begin{remark}
	Observe that the residue classes appearing in the formulas for $Z_{\frak{g}}(q)$ below are always exponents of the corresponding Lie algebra, namely: $\{ 1, N-1\}$ for $D_{N+2}$,
	$\{1,5\}$ for $E_6$, $\{1,5,7,11\}$ for $E_7$, and  $\{1,11,19,29 \}$ for $E_8$ (see \cite[page 299]{OV}). It would be interesting to find an explanation for this numerical coincidence.
\end{remark}

\subsection{$D_{N+2}$}

We begin with the case $\frak{g} = D_{N+2}$ for $N \geq 2$.
\begin{proposition}
\label{P:DN+2}
\begin{enumerate}[leftmargin=*, label=\textnormal{(\arabic*)}]
	\item For $N$ odd, we have
	\begin{equation}
	\label{DF}
	\begin{aligned}
	Z_{D_{N+2}}(q) & =q^{- \frac{N}{2}- \frac{1}{2N}} \left(F_{-1,N}(2N\tau) + F_{2N-1,N}(2N \tau)\right) \\
	& = -2q^{-\frac{N}{2}}\sum_{m\equiv 1\pmod{2N}}\sgn(m)q^{\frac{m^2-1}{2N}}+2q^{-\frac{N}{2}}.
	\end{aligned}
	\end{equation}
	As $t \to 0^+$, we have
	\begin{equation}\label{Das}
	Z_{D_{N+2}}\left(e^{-2\pi t}\right) \sim \frac{2}{N}.
	\end{equation}
	\item  For $N$ even, we have
	\begin{equation*}
	\begin{aligned}
	Z_{D_{N+2}}(q) & =q^{-\frac{N}{2}-\frac{1}{2N}}\left(F_{-1,N}(2N\tau)+2 F_{N+1,N}(2N\tau)+F_{2N-1,N}(2N\tau)\right) \\
	& = -2q^{-\frac{N}{2}}\sum_{m\equiv 1, N-1\pmod{2N}}\sgn(m)q^{\frac{m^2-1}{2N}}+2q^{-\frac{N}{2}}.
	\end{aligned}
	\end{equation*}
	As $t\to0^+$, we have
	\begin{equation*}
	Z_{D_{N+2}}\left(e^{-2\pi t}\right)\sim -2\pi t.
	\end{equation*}
\end{enumerate}
\noindent In particular $q^{\frac{1}{4}+\frac{1}{4N^2}}Z_{D_{N+2}}(q^{\frac{1}{2N}})$ is a quantum modular form with quantum set $\Q$.
\end{proposition}
\begin{proof}
	For $D_{N+2}$, $N \geq 2$, the graph is

\begin{figure}[h]
	\begin{center}
		\begin{tikzpicture}
		\node[shape=circle,fill=black, scale = 0.4] (1) at (-2.0,0) { };
		\node[shape=circle,fill=black, scale = 0.4] (2) at (-0.5,0) { };
		\node[shape=circle,fill=black, scale = 0.4] (3) at (1,0) { };
		\node[shape=circle,fill=black, scale = 0.4] (4) at (2,-1)  { };
		\node[shape=circle,fill=black, scale = 0.4] (5) at (2,1) { };
		\node[shape=circle,fill=black, scale = 0.4] (6) at (-3.5,0) { };
		
		\node[draw=none] (B1) at (-3.5,0.4) {$a_1$};
		\node[draw=none] (B2) at (-2,0.4) {$a_2$};
		\node[draw=none] (B22) at (-0.5,0.4) {$a_{N-1}$};
		\node[draw=none] (B3) at (0.85,0.4) {$a_{N}$};
		\node[draw=none] (B4) at (2.75,-1) {$a_{N+2}$};
		\node[draw=none] (B5) at (2.6,1) {$a_{N+1}$};
		
		\path [-] (6) edge node[left] {} (1);
		\path [dotted] (1) edge node[left] {} (2);
		\path [-] (2) edge node[left] {} (3);
		\path [-](3) edge node[left] {} (4);
		\path [-](3) edge node[left] {} (5);
		\end{tikzpicture}
	\end{center}
\caption{Labeled Dynkin diagram for $D_{N+2}$.} \label{F:DN+2}
\end{figure}
\noindent Thus we have the following matrix associated to $D_{N+2}$ (recalling that the Cartan matrix has $a_j = 2$ for all $j$)
$$
M_{N+2} := \left(
\begin{smallmatrix}
2 & -1 & & & & \\
-1 & 2 & - 1 & & & \\
& -1 & 2 & -1 & & & \\
& & & \ddots & & & \\
& & & & 2 & -1 & -1 \\
& & & & -1 & 2 & 0 \\
& & & & -1 & 0 & 2
\end{smallmatrix}
\right).
$$
By \eqref{E:Zq=CT}, we need to compute
\begin{equation}
\label{E:CT3star}
{\rm CT}_{\b w} \left(\Theta_{M}(\b w) \frac{\prod_{r\in\{1,N+1,N+2\}}\left(w_r-w_r^{-1}\right)}{w_N-w_N^{-1}}\right).
\end{equation}
It is known that $\det(M_{N+2}) = 4$ for all $N$, and that the adjucate matrix is (see Table 2 on page 295 of \cite{OV} for this and all other simple Lie algebras)
\begin{equation*}
{\rm adj}(M_{N+2}) =
\left(\begin{smallmatrix}
4 & 4 & 4 & \cdots & 4 & 2 & 2 \\
4 & 8 & 8 & \cdots & 8 & 4 & 4 \\
4 & 8 & 12 & \cdots & 12 & 6 & 6 \\
\vdots & \vdots & \vdots & \ddots & \vdots & \vdots & \vdots \\
4 & 8 & 12 & \cdots & 4N & 2N & 2N \\
2 & 4 & 6 & \cdots & 2N & N+2 & N \\
2 & 4 & 6 & \cdots & 2N & N & N+2
\end{smallmatrix}\right).
\end{equation*}
As explained in the proof of Lemma \ref{help}, if we write the theta function using  \eqref{E:Theta2}, we obtain the restriction ${\rm adj}(M_{N+2}) \b{m} \equiv 0 \pmod{4}.$

We evaluate the constant terms in $w_2, \dots, w_{N-1}$ in \eqref{E:CT3star}, which means that $m_2, \dots, m_{N-1}$ are set to zero. It is then clear that the system ${\rm adj}(M_{N+2}) \b{m} \equiv 0 \pmod{4}$ is equivalent to the restriction along the rows $1, N, N+1,$ and $N+2$. Thus we have $L {\boldsymbol \ell} \equiv 0 \pmod{4},$ where
\begin{equation*}
{\boldsymbol \ell} := (m_1, m_N, m_{N+1}, m_{N+2})\quad \textnormal{ and } \quad L
 := \left(\begin{smallmatrix}
4 & 4 & 2 & 2 \\
4 & 4N & 2N & 2N \\
2 & 2N & N+2 & N \\
2 & 2N & N & N+2
\end{smallmatrix}\right).
\end{equation*}

We first assume that $N$ is odd. Then the congruence conditions reduce to $m_{N+1} \equiv m_{N+2} \pmod{2}$ and $2m_1 + 2m_N + Nm_{N+1} + (N+2)m_{N+2} \equiv 0 \pmod{4}$.  Plugging the first equation into the second, we obtain $m_N \equiv m_1 + \frac12(m_{N+1} - m_{N+2}) \pmod{2}$. Thus evaluating the constant terms with respect to $w_1, w_{N+1}$, and $w_{N+2}$, \eqref{E:CT3star} becomes
\begin{equation}
\label{E:CTwN}
{\rm CT}_{w_N} \left(\sum_{\substack{m_1, m_{N+1}, m_{N+2} \in \{-1, 1\} \\ m_N \equiv m_1 + \frac{m_{N+1} - m_{N+2}}{2} \pmod{2}}} \mathcal{A}(\boldsymbol{\ell})w_N^{m_N}\right),
\end{equation}
where
\begin{align*}
\label{E:Q2DN}
Q_2(\boldsymbol{\ell}) &= \boldsymbol{\ell}^T \left(\tfrac14 L\right) \boldsymbol{\ell} \\\notag
 &= m_1^2 + N m_N^2 + \frac{N+2}{4} \left(m_{N+1}^2 + m_{N+2}^2\right)
+ m_N(2m_1 + Nm_{N+1} + Nm_{N+2}) \\\notag
& \hspace{9cm}  + m_1 (m_{N+1} + m_{N+2}) + \frac{N}{2} m_{N+1} m_{N+2}
\end{align*}
and
\begin{equation*}
\mathcal{A}(\boldsymbol{\ell} ) := -m_1 m_{N+1} m_{N+2} q^{\frac12 Q_2(\boldsymbol{\ell})} \frac{1}{w_N - w_N^{-1}}.
\end{equation*}
Note that $Q_2$ is symmetric in $m_{N+1}$ and $m_{N+2}$.

Since $\mathcal{A}(-\boldsymbol{\ell} ) = -\mathcal{A}(\boldsymbol{\ell} )$, if we group $(m_1, m_{N+1}, m_{N+2})$ with its negative, thus \eqref{E:CTwN} becomes
\begin{align*}
{\rm CT}_{w_N} & \left( \left( \sum_{\substack{m_1=m_{N+1}=m_{N+2}=-1\\ m_N \equiv 1 \pmod{2}}} \hspace{-1cm}
q^{\frac12 \left(Nm_N^2 - 2(N+1)m_N + N+4\right)} +
\sum_{\substack{m_1=-m_{N+1}=-m_{N+2}=-1 \\ m_N \equiv 1 \pmod{2}}} \hspace{-1cm}
q^{\frac12 \left(Nm_N^2 + 2(N-1)m_N + N \right)}  \right.  \right.\\
& \left.\left. \hspace{4.4cm} +
2 \sum_{\substack{m_1=-m_{N+1}=m_{N+2}=1 \\ m_N \equiv 0 \pmod{2}}}
q^{\frac12 \left(Nm_N^2 + 2m_N +\frac12 \right)}
\right) \frac{w_N^{m_N} - w_N^{-m_N}}{w_N - w_N^{-1}}\right),
\end{align*}
where the third sum is doubled to account for both $(m_1, m_{N+1}, m_{N+2}) = (1,-1,1)$ and $(1,1,-1)$.
By \eqref{E:CT=sgno}, the third sum does not contribute to the constant term, leaving only the contributions from the first two sums. Using \eqref{E:sgn=F}, we obtain
\begin{multline*}
\sum_{m \equiv 1 \pmod{2}} \sgn(m) q^{\frac12 \left(Nm^2 - 2(N+1)m + N+4 \right)} +
\sum_{m \equiv 1 \pmod{2}} \sgn(m) q^{\frac12 \left(Nm^2 + 2(N-1)m + N \right)} \\
= q^{1 - \frac{1}{2N}} F_{-1,N}(2N\tau) + q^{1 - \frac{1}{2N}} F_{2N-1,N}(2N\tau).
\end{multline*}
Including the additional $q^{-\frac{N+2}{2}}$ from \eqref{E:Zq=CT}, we therefore have the first expression in \eqref{DF}; the second expression in \eqref{DF} follows from \eqref{rel1} and \eqref{rel2}.

Recalling \eqref{E:Fasymp} yields the asymptotic behavior in \eqref{Das}.

Next suppose that $N$ is even. We renumber rows and columns to match the conventions of Proposition \ref{P:quantum}, letting $\boldsymbol{\ell} = (m_1, m_2, m_3, m_4)$, and
$$
A_{N+2} := \left(\begin{smallmatrix}
1 & \frac12 & \frac12 & 1 \\ \frac12 & \frac{N+2}{4} & \frac{N}{4} & \frac{N}{2} \\ \frac12 & \frac{N}{4} & \frac{N+2}{4} & \frac{N}{2} \\ 1 & \frac{N}{2} & \frac{N}{2} & N
\end{smallmatrix}\right).
$$
Inspecting the corresponding matrices gives that Corollary \ref{C:Z=ZA} may be used, and a short additional calculation shows that Proposition \ref{P:quantum} applies. Plugging in the entries of $A_{N+2}$ and employing \eqref{rel1}, \eqref{rel2}, and \eqref{E:Fasymp} then directly gives the claimed series and asymptotic expressions.
\end{proof}

\subsection{$E$-series}
For $\frak{g} = E_N$ ($N \in \{6,7,8\}$) we enumerate the vertices as in Figure \ref{F:EN}.

\begin{figure}[h]
\begin{center}

	\begin{tikzpicture}[thick, scale=0.8]
	\node[shape=circle,fill=black, scale = 0.4] (3) at (-1,0) { };
	\node[shape=circle,fill=black, scale = 0.4] (4) at (0,0) { };
	\node[shape=circle,fill=black, scale = 0.4] (5) at (1,0) { };
	\node[shape=circle,fill=black, scale = 0.4] (6) at (2,0)  { };
	\node[shape=circle,fill=black, scale = 0.4] (7) at (3,0) { };
	\node[shape=circle,fill=black, scale = 0.4] (8) at (1,1) { };

	\node[draw=none] (A1) at (-1,-0.4) {$a_1$};
	\node[draw=none] (A2) at (0,-0.4) {$a_{2}$};
		\node[draw=none] (A2) at (1,-0.4) {$a_{3}$};
	
	\node[draw=none] (A3) at (1,1.3) {$a_6$};
	\node[draw=none] (A4) at (2,-0.4) {$a_{4}$};
	\node[draw=none] (A4) at (3,-0.4) {$a_{5}$};
	
	\path [-](3) edge node[left] {} (4);
	
	\path [-] (4) edge node[left] {} (5);
	\path [-] (5) edge node[left] {} (6);
	\path [-] (5) edge node[left] {} (7);
	\path [-] (5) edge node[left] {} (8);
	\end{tikzpicture}
	\qquad
	\begin{tikzpicture}[thick, scale=0.8]
	\node[shape=circle,fill=black, scale = 0.4] (2) at (-2,0) { };
	\node[shape=circle,fill=black, scale = 0.4] (3) at (-1,0) { };
	\node[shape=circle,fill=black, scale = 0.4] (4) at (0,0) { };
	\node[shape=circle,fill=black, scale = 0.4] (5) at (1,0) { };
	\node[shape=circle,fill=black, scale = 0.4] (6) at (2,0)  { };
	\node[shape=circle,fill=black, scale = 0.4] (7) at (3,0) { };
	\node[shape=circle,fill=black, scale = 0.4] (8) at (1,1) { };
	
	\node[draw=none] (A1) at (-2,-0.4) {$a_1$};
	\node[draw=none] (A1) at (-1,-0.4) {$a_2$};
	\node[draw=none] (A1) at (0,-0.4) {$a_3$};
	\node[draw=none] (A2) at (1,-0.4) {$a_{4}$};
	\node[draw=none] (A3) at (1,1.3) {$a_7$};
	\node[draw=none] (A4) at (2,-0.4) {$a_{5}$};
	\node[draw=none] (A4) at (3,-0.4) {$a_{6}$};

	\path [-](2) edge node[left] {} (3);
	\path [-](3) edge node[left] {} (4);
	
	\path [-] (4) edge node[left] {} (5);
	\path [-] (5) edge node[left] {} (6);
	\path [-] (5) edge node[left] {} (7);
	\path [-] (5) edge node[left] {} (8);
	\end{tikzpicture}
\qquad
\begin{tikzpicture}[thick, scale=0.8]
	\node[shape=circle,fill=black, scale = 0.4] (1) at (-3,0) { };
	\node[shape=circle,fill=black, scale = 0.4] (2) at (-2,0) { };
	\node[shape=circle,fill=black, scale = 0.4] (3) at (-1,0) { };
	\node[shape=circle,fill=black, scale = 0.4] (4) at (0,0) { };
	\node[shape=circle,fill=black, scale = 0.4] (5) at (1,0) { };
	\node[shape=circle,fill=black, scale = 0.4] (6) at (2,0)  { };
	\node[shape=circle,fill=black, scale = 0.4] (7) at (3,0) { };
	\node[shape=circle,fill=black, scale = 0.4] (8) at (1,1) { };

	\node[draw=none] (A1) at (-3,-0.4) {$a_1$};
	\node[draw=none] (A1) at (-2,-0.4) {$a_2$};
	\node[draw=none] (A1) at (-1,-0.4) {$a_3$};
	\node[draw=none] (A1) at (0,-0.4) {$a_4$};	
	\node[draw=none] (A2) at (1,-0.4) {$a_{5}$};
	\node[draw=none] (A3) at (1,1.3) {$a_8$};
	\node[draw=none] (A4) at (2,-0.4) {$a_{6}$};
	\node[draw=none] (A4) at (3,-0.4) {$a_{7}$};

	\path [-] (1) edge node[left] {} (2);
	\path [-](2) edge node[left] {} (3);
	\path [-](3) edge node[left] {} (4);
	
	\path [-] (4) edge node[left] {} (5);
	\path [-] (5) edge node[left] {} (6);
	\path [-] (5) edge node[left] {} (7);
	\path [-] (5) edge node[left] {} (8);
	\end{tikzpicture}

\end{center}
\caption{Dynkin diagrams for $E_6, E_7,$ and $E_8$.} \label{F:EN}
\end{figure}
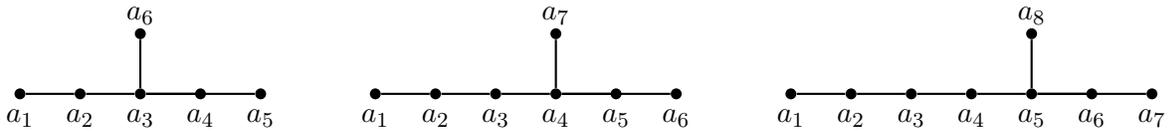

\normalsize
By \eqref{E:Zq=CT}, we need to compute
\begin{align}
\label{E:ZE}
&\text{CT}_{\b w}\left(\Theta_M(\b w)\frac{\prod_{r\in\{1,N-1,N\}}\left(w_r-w_r^{-1}\right)}{w_{N-3}-w_{N-3}^{-1}}\right).
\end{align}
We treat each case separately.

\subsubsection{$E_6$} In the first case, we use Theorem \ref{T:Zqquantum} in order to calculate the invariant series.
\begin{proposition}
\label{P:E6}
	We have
	\begin{align}\label{ZTE6}
	Z_{E_6}(q) &= q^{-\frac{25}{12}} \left(F_{-1,6}(12 \tau)+ F_{7,6}(12\tau)\right) = -q^{-2}\sum_{m\equiv 1,5\pmod{12}}\sgn(m)q^{\frac{m^2-1}{12}}+2q^{-2}.
	\end{align}
	In particular $q^{\frac{25}{144}}Z_{E_6}(q^{\frac{1}{12}})$ is a quantum modular form with quantum set $\Q$.
	
	As $t \to 0^+$, we have
\begin{align}\label{asE6}
	Z_{E_6}\left(e^{-2\pi t}\right) \sim 1.
\end{align}
\end{proposition}
\begin{proof}
We have the matrix
$$
M := \left( \begin{smallmatrix} 2&-1&0&0&0&0\\ -1&2&-1&0&0&0\\ 0&-1&2&-1&0&-1\\ 0&0&-1&2&-1&0\\ 0&0&0&-1&2&0\\ 0&0&-1&0&0&2
\end{smallmatrix} \right),
$$
which has $\det(M) = 3$, and the adjucate matrix
$$
{\rm adj}(M) = \left( \begin{smallmatrix} 4&5&6&4&2&3\\ 5&10&12&8&4&6\\ 6&12&18&12&6&9\\ 4&8&12&10&5&6\\ 2&4&6&5&4&3\\ 3&6&9&6&3&6\end{smallmatrix} \right).
$$

Proceeding as in the proof of Theorem \ref{T:Zqquantum}, we evaluate the constant terms in all variables except $w_3$ in \eqref{E:ZE}, obtaining (taking into account the different vertex labeling in \eqref{E:CTw4})
\begin{equation*}
Z_{E_6}(q) = -q^{-\frac32}\operatorname{CT}_{w_3}\left(\sum_{\substack{m_1,m_5,m_6\in\{\pm 1\} \\ \b m = (m_1,0,m_3,0,m_5,m_6)^T\in M\Z^6}} m_1m_5m_6 q^{Q_2(\b m)}\frac{w_3^{m_3}}{w_3-w_3^{-1}}\right).
\end{equation*}
We observe that Corollary \ref{C:Z=ZA} applies (although Proposition \ref{P:Z3star} does not), since if $m_2 = m_4 = 0$, the congruence restriction ${\rm adj}(M) \b{m} \equiv 0 \pmod{3}$ is equivalent to $L \b{m} \equiv 0 \pmod{3}$, where
$$
\boldsymbol{\ell
} := (m_1, m_3, m_5, m_6)\quad \textnormal{ and }\quad L := \left( \begin{smallmatrix} 4 & 6 & 2 & 3 \\ 6 & 18 & 6 & 9 \\ 2 & 6 & 4 & 3 \\ 3 & 9 & 3 & 6 \end{smallmatrix} \right).
$$
We may therefore suppress $m_2$ and $m_4$.
This system of congruences further reduces to the single restriction $m_1 \equiv m_5 \pmod{3}$, and with $m_1, m_5, m_6 \in \{\pm 1\}$, this implies $m_5 = m_1$. Thus
\begin{equation*}
Z_{E_\ell}(q) = -q^{-\frac32} {\rm CT}_{w_3}\left( \sum_{\substack{m_1=m_5, m_6 \in \{-1, 1\} \\ m_3 \in \Z}} m_6
q^{\frac12 Q_2(\boldsymbol{\ell})} \frac{w_3^{m_3}}{w_3 - w_3^{-1}}\right),
\end{equation*}
where
\begin{align*}
Q_2(\boldsymbol{\ell}) & := \boldsymbol{\ell}^T \left(\tfrac13 L\right) \boldsymbol{\ell}.
\end{align*}
Setting $m_5=m_1$, this simplifies to
\begin{align*}
Q_2(\boldsymbol{\ell}) & = 6m_3^2 + (8m_1 + 6m_6)m_3 + 4m_1 m_6 + 6.
\end{align*}

Changing for the terms $m_1=m_6=1$ and $m_1=-m_6=-1$ $m_3$ into $-m_3$, we obtain
\begin{align*}
{\rm CT}_{w_3} & \left( \left(\sum_{\substack{m_1=m_6=-1\\ m_3 \in \Z}} q^{\frac12 \left(6m_3^2 - 14m_3 + 10\right)}
+\sum_{\substack{m_1=-m_6=1\\ m_3 \in \Z}} q^{\frac12 \left(6m_3^2 + 2m_3 + 2\right)} \right) \frac{w_3^{m_3} - w_3^{-m_3}}{w_3 - w_3^{-1}}\right) \\
& = \sum_{m \equiv 1 \pmod{2}} \sgn(m) \left(q^{3m^2 - 7m + 5} + q^{3m^2 + m + 1}\right),
\end{align*}
using \eqref{E:CT=sgno}.
Using \eqref{E:sgn=F}, this gives the first expression in \eqref{ZTE6}; the second expression follows from \eqref{rel1} and \eqref{rel2}.

The asymptotic formula \eqref{asE6} may be concluded from \eqref{E:Fasymp}.
\end{proof}

\subsubsection{$E_7$}

For $E_7$ and $E_8$ we use Proposition \ref{P:quantum} and Corollary \ref{C:Z=ZA}.
\begin{proposition}
\label{P:E7}
We have
\begin{equation}
\label{ZE7}
\begin{aligned}
Z_{E_7}(q) &=q^{-\frac{61}{24}}\left(F_{-1,12}(24\tau)+ F_{19,12}(24\tau) + F_{17,12}(24\tau) + F_{13,12}(24\tau)\right) \\
& = -q^{-\frac52}\sum_{m\equiv 1,5,7,11\pmod{24}}\sgn(m)q^{\frac{m^2-1}{24}}+2q^{-\frac52}.
\end{aligned}
\end{equation}
In particular $q^{\frac{61}{576}} Z_{E_7}(q^{\frac{1}{24}})$ is a quantum modular form.

As $t \to 0^+$, we have
\begin{align}\label{as7}
	Z_{E_7}\left(e^{-2\pi t}\right) \sim -4\pi t.
\end{align}
\end{proposition}
\begin{proof}
We have the matrix
$$
M := \left(\begin{smallmatrix} 2&-1&0&0&0&0&0\\ -1&2&-1&0&0&0&0\\ 0&-1&2&-1&0&0&0\\ 0&0&-1&2&-1&0&-1\\ 0&0&0&-1&2&-1&0\\ 0&0&0&0&-1&2&0\\ 0&0&0&-1&0&0&2
\end{smallmatrix}\right),
$$
which has $\det(M) = 2$, and the adjucate matrix
$$
{\rm adj}(M) = \left(\begin{smallmatrix} 3&4&5&6&4&2&3\\ 4&8&10&12&8&4&6\\ 5&10&15&18&12&6&9\\ 6&12&18&24&16&8&12\\ 4&8&12&16&12&6&8\\ 2&4&6&8&6&4&4\\ 3&6&9&12&8&4&7\end{smallmatrix}\right).
$$
We can immediately evaluate the constant terms in $w_2, w_3$ and $w_5$ in \eqref{E:ZE}, and the congruence restriction ${\rm adj}(M) \b{m} \equiv 0 \pmod{2}$ reduces to $L \boldsymbol{\ell} \equiv 0 \pmod{2}$, where
$$
\boldsymbol{\ell} := (m_1, m_4, m_6, m_7)\quad \textnormal{ and } \quad L := \left(\begin{smallmatrix} 3 & 6 & 2 & 3 \\ 6 & 24 & 8 & 12 \\ 2 & 8 & 4 & 4 \\ 3 & 12 & 4 & 7 \end{smallmatrix}\right).
$$
This implies that Corollary \ref{C:Z=ZA} applies. Swapping the second and fourth rows, we then use Proposition \ref{P:quantum} with
$$
A := \left(\begin{smallmatrix} \frac32 & 1 & \frac32 & 3 \\ 1 & 2 & 2 & 4 \\ \frac32 & 2 & \frac72 & 6 \\ 3 & 4 & 6 & 12 \end{smallmatrix}\right),
$$
and a short calculation gives the first expression in \eqref{ZE7}; the second expression follows from \eqref{rel1} and \eqref{rel2}.

The asymptotic behavior \eqref{as7} may be concluded  from \eqref{AsZA}.
\end{proof}

\subsubsection{$E_8$} Note that $\widehat{Z}_{0}(q)$ for this graph was also calculated numerically in \cite[(3.155)]{GPPV}, and is our only explicit example with a unimodular matrix.

\begin{proposition}
\label{P:E8}
We have
\begin{equation}
\label{ZE8}
	\begin{aligned}
		Z_{E_8}(q) &=q^{-\frac{181}{60}} \left(F_{-1,30} (60 \tau) + F_{49,30}(60\tau) + F_{41,30}(60\tau)+F_{31,30}(60\tau) \right) \\
		& = -q^{-3} \sum_{m \equiv 1, 11, 19, 29 \pmod{60}} \sgn(m) q^{\frac{m^2 - 1}{60}} + 2q^{-3}.
	\end{aligned}
\end{equation}
	In particular $q^{\frac{181}{3600}} Z_{E_8}(q^{\frac{1}{60}})$ is a quantum modular form.
	
	As $t \to 0^+$, we have
	\begin{align}\label{as8}
		Z_{E_8}\left(e^{-2\pi t}\right) \sim -8 \pi t.
	\end{align}
\end{proposition}
\begin{proof}
The Cartan matrix is
$$
M := \left(\begin{smallmatrix}
2 &  -1 & 0 & 0 & 0 & 0 & 0 & 0 \\
-1 & 2 & -1 & 0 & 0 & 0 & 0 & 0 \\
0 & -1 & 2 & -1 & 0 & 0 & 0 & 0 \\
0 & 0 & -1 & 2 & -1 & 0 & 0 & 0 \\
0 & 0 & 0 & -1 & 2 & -1 & 0 & -1 \\
0 & 0 & 0 & 0 & -1 & 2 & -1 & 0 \\
0 & 0 & 0 & 0 & 0 & -1 & 2 & 0 \\
0 & 0 & 0 & 0 & -1 & 0 & 0 & -1
\end{smallmatrix}\right),
$$
which has $\det(M) = 1$ (thus Corollary \ref{C:Z=ZA} and Proposition \ref{P:quantum} are guaranteed to apply), and
$$
{\rm adj}(M) = M^{-1}=\left(
\begin{smallmatrix} 2 &  3 & 4 & 5 & 6 &  4 &  2 &  3 \\
3 &  6 & 8 & 10 & 12 & 8 & 4 & 6  \\
4 & 8 & 12 & 15 & 18 & 12 & 6 & 9  \\
5 &10 & 15 & 20 & 24 & 16 & 8 & 12 \\
6 &12 & 18 & 24 & 30 & 20 & 10 & 15 \\
4 & 8  & 12 &  16 &  20 &  14 &  7 &  10 \\
2 &  4 &  6 & 8 &  10 &  7 & 4 &  5\\
3 &  6 &  9 & 12 & 15 &  10 &  5 &  8
\end{smallmatrix}
\right).
$$

Extracting the constant terms with respect to $w_2,w_3,w_4,$ and $w_6$ in \eqref{E:ZE}, we are left with the minor along rows $1, 5, 7,$ and $8$ of $M^{-1}$. Further rearranging these rows and columns, we now apply Proposition \ref{P:quantum} with
$$
A := \left( \begin{smallmatrix} 2 & 2 & 3 & 6 \\ 2 & 4 & 5 & 10 \\ 3 & 5 & 8 & 15 \\ 6 & 10 & 15 & 30 \end{smallmatrix} \right).
$$
This directly gives the first expression in \eqref{ZE8}; the second identity follows from \eqref{rel1} and \eqref{rel2}.

The asymptotics in \eqref{AsZA} implies \eqref{as8}.
\end{proof}

\begin{remark}
Lawrence and Zagier studied a very similar series in \cite{LZ}, namely
\begin{equation*}
A(q) := \sum_{m\geq 1} \chi_+(m) q^{\frac{m^2-1}{120}}
\end{equation*}
with
\begin{equation*}
\chi_+(m) := \begin{cases}
1\quad &\textnormal{ if } m\equiv 1,11,19,29\pmod{60},\\
-1\quad &\textnormal{ if } m\equiv 31,41,49,59\pmod{60}.
\end{cases}
\end{equation*}	
\end{remark}

Comparing to \eqref{ZE8}, we see that $q^3 Z_{E_8}(q) = 2 - A(q^2)$. Furthermore, Theorem 2 of \cite{LZ} shows that $2-A(q)$ is also the rescaled WRT-invariant of the Poincar\'{e} homology sphere.

\section{General leg star graphs and the proof of Theorem \ref{T:main} \ref{T:main:nleg}}\label{genstar}

In this section we extend the ideas used for $3$-leg star graphs in order to prove the quantum modularity of $Z(q)$ for $(\ell-1)$-leg star graphs with $\ell \geq 4$ (recall Figure \ref{F:ellstar}).

\subsection{Quantum modularity for $(\ell-1)$-leg star graphs.}

The main result of this section proves Theorem \ref{T:main} \ref{T:main:nleg} by giving a more precise description of the quantum modular forms that arise in $Z(q)$ for $(\ell-1)$-leg star graphs.

\begin{theorem} \label{n-star} For any $(\ell-1)$-leg star graph, with $\ell \geq 4$, $Z(q)$ is a linear combination of derivatives of (up to $q$-powers) quantum modular forms of weight $\frac12$ or $\frac32$.
\end{theorem}
\begin{remark}
	The quantum modular forms in Theorem \ref{n-star} do not necessarily have the same quantum set.
\end{remark}

Before proving Theorem \ref{n-star}, we require some auxiliary results.
As in Section \ref{S:Sing}, we have
\begin{align*}
Z(q)=\frac{q^{\frac{-3N + \sum_{\nu = 1}^N a_\nu}{2}}}{(2\pi i)^N}\text{PV}\int_{|w_\ell|=1}\frac{1}{\left(w_\ell-w_\ell^{-1}\right)^{\ell-3}}
h(w_\ell)\frac{dw_\ell}{w_\ell},
\end{align*}
where
$$
h(w_\ell):=\int_{|w_k|=1}\Theta_M(\b w)\prod_{r=1}^{\ell-1}\left(w_r-w_r^{-1}\right)\prod_{\substack{1 \leq k \leq N \\ k \neq \ell}}\frac{dw_k}{w_k}.
$$

\begin{proposition}
\label{P:Zlstar}
For $(\ell-1)$-leg star graphs, we have
\begin{equation*}
Z(q) =q^{\frac{-3N + \sum_{\nu = 1}^N a_\nu}{2}} \textnormal{CT}_{\b w}\left(\frac{	 \Theta_M(\boldsymbol{w})\prod_{r = 1}^{\ell-1}\left(w_r-w_r^{-1}\right)}{\left(w_\ell-w_\ell^{-1}\right)^{\ell-3}}\right).
\end{equation*}
\end{proposition}

\begin{proof}

The statement amounts to the claim that the residues of $h(w_{\ell})(w_{\ell} - w_{\ell}^{-1})^{-\ell+3}$ at $w_{\ell} = \pm 1$ vanish. Plugging in \eqref{E:Theta2} and observing that the only poles in the integrand are $w_k = 0$, we have
\begin{align*}
h(w_\ell) &= {\rm CT}_{w_1, \cdots, w_{\ell-1}, w_{\ell+1}, \cdots, w_N}
\left(\sum_{\b m \in M \Z^N} q^{Q_2(\b{n})} e^{2 \pi i \b{w}^T \b{z}} \prod_{r = 1}^{\ell - 1} \left(w_r - w_r^{-1}\right)\right).
\end{align*}
Now \eqref{E:CTw-w^-1} gives
\begin{align*}
h(w_\ell) &= (-1)^{\ell-1} \sum_{\substack{m_1, \cdots, m_{\ell-1} \in \{\pm 1\} \\ \b{m} = (m_1, \cdots, m_\ell, 0, \cdots, 0)^T \in M \Z^N }}  m_1 \cdots m_{\ell-1} \, q^{Q_2(\b{m})} w_\ell^{m_\ell}.
\end{align*}
Corollary \ref{C:CTn} then implies that
\begin{align*}
{\rm CT}_{w_\ell} \left(\frac{h(w_\ell)}{ \left(w_\ell - w_\ell^{-1}\right)^{\ell-3}}\right)
= {\rm CT}_{w_\ell} \left(\frac{H(w_\ell)}{w_\ell - w_\ell^{-1}}\right),
\end{align*}
where
\begin{equation}
\label{E:gelldef}
H(w_\ell) := \frac{(-1)^{\ell-1}}{(\ell-4)!} \sum_{\substack{m_1, \cdots, m_{\ell-1} \in \{\pm 1\} \\ \b{m} = (m_1, \cdots, m_\ell, 0, \cdots, 0)^T \in M \Z^N }} m_1 \cdots m_{\ell-1} \left(\frac{m_\ell-\ell+5}{2}\right)_{\ell-4} \, q^{Q_2(\b{m})} w_\ell^{m_\ell-\ell+4}.
\end{equation}
As in the proof of Theorem \ref{T:Zqquantum}, Lemma \ref{help} gives that $\b{m} \in M \Z^N$ if and only if $-\b{m} \in M \Z^N$. We can therefore change $\b{m} \mapsto -\b{m}$ and use \eqref{E:Poch} in order to obtain
\begin{equation*}
H\left(w_\ell^{-1}\right) = (-1)^{\ell-1} (-1)^{\ell-4} H(w_\ell) = - H(w_\ell).
\end{equation*}
The proposition statement follows from Lemma \ref{L:removable}.
\end{proof}

We are now ready to prove Theorem \ref{n-star}.
\begin{proof}[Proof of Theorem \ref{n-star}]
We follow the basic approach from the proof of Theorem \ref{T:Zqquantum}. Applying Proposition \ref{P:Zlstar} and recalling \eqref{E:gelldef}, we have
\begin{align}
\notag
Z(q)
& = \frac{q^{\frac{-3N + \sum_{\nu = 1}^N a_\nu}{2}} (-1)^{\ell-1}}{(\ell-4)!}
\\
&\quad \times {\rm CT}_{w_\ell} \left(\sum_{\substack{m_1, \cdots, m_{\ell-1} \in \{\pm 1\} \\ \b{m} = (m_1, \cdots, m_\ell, 0, \cdots, 0)^T \in M \Z^N }} \!\!\! m_1 \cdots m_{\ell-1} \! \left(\frac{m_\ell - \ell + 5}{2}\right)_{\ell-4}
q^{Q_2(\b{m})} \frac{w_\ell^{m_\ell - \ell + 4}}{w_\ell - w_\ell^{-1}}\!\right)\!.
\label{E:CTwell}
\end{align}

Lemma \ref{help} again implies that there exists an additive subgroup $S \subset (\Z / D \Z)^N$ such that $\b{m} \in M \Z^N$ if and only if $\b{m} \equiv \b{s} \pmod{D}$ for some $\b{s} \in S$. Let $T \subset S$ be the subset of elements of the form $\b{s} = (\varepsilon_1, \cdots, \varepsilon_{\ell-1}, g, 0, \cdots, 0)^T \in S$, where $\varepsilon_j \equiv \pm 1 \pmod{D}$, and $g$ is a residue modulo $D$.
We can therefore write the constant term from \eqref{E:CTwell} as (ignoring outside constants)
\begin{align}
\label{E:CTwell=j}
& {\rm CT}_{w_\ell} \left(\sum_{\b{s} \in T} \; \sum_{\substack{ \b{m} = (\varepsilon_1, \cdots, \varepsilon_{\ell-1}, m_\ell, 0, \cdots, 0)^T \\
m_\ell \equiv g \pmod{D}}} \varepsilon_1 \cdots \varepsilon_{\ell-1} \, \left(\frac{m_\ell - \ell + 5}{2}\right)_{\ell-4} q^{Q_2(\b{m})} \frac{w_\ell^{m_\ell -\ell + 4}}{w_\ell - w_\ell^{-1}}\right).
\end{align}
As before, $T = -T$. Combining with \eqref{E:Poch}, we pair the $\b{s}$ and $-\b{s}$ terms, so \eqref{E:CTwell=j} becomes
\begin{align*}
& \frac12{\rm CT}_{w_\ell} \left(\sum_{\b{s} \in T}  \sum_{\substack{ \b{m} = (\varepsilon_1, \cdots, \varepsilon_{\ell-1}, m_\ell, 0, \cdots, 0)^T \\
m_\ell \equiv g \pmod{D}}} \varepsilon_1 \cdots \varepsilon_{\ell-1} \left(\frac{m_\ell - \ell + 5}{2}\right)_{\ell-4} q^{Q_2(\b{m})} \frac{w_\ell^{m_\ell} - w_\ell^{-m_\ell}}{w_\ell^{\ell-4} \left(w_\ell - w_\ell^{-1}\right)}\right).
\end{align*}
We now apply \eqref{E:CTmn} to the inner sum for a particular fixed $\b{s} = \left(\varepsilon_1, \cdots, \varepsilon_{\ell-1}, g, 0, \cdots, 0\right)^T \in T$, which gives a sum of the form (ignoring all outside constants and $q$-powers)
\begin{align*}
& {\rm CT}_{w} \left(\sum_{m \equiv g \pmod{D}} \left(\frac{m - \ell + 5}{2}\right)_{\ell-4} q^{Q_2(\b{m})} \frac{w^{m} - w^{-m}}{w^{\ell-4} \left(w - w^{-1}\right)}\right) \\
& = \!\!\!\!\sum_{\substack{m \equiv g \pmod{D} \\ m \equiv \ell-3 \pmod{2}}} \!\!\!\left(\frac{m - \ell + 5}{2}\right)_{\ell-4} \delta_{|m| \geq \ell-3} \ \sgn(m)q^{Q_2(\b{m})} = \!\!\!\!\sum_{\substack{m \equiv g \pmod{D} \\ m \equiv \ell-3 \pmod{2}}} \!\!\!\left(\frac{m - \ell + 5}{2}\right)_{\ell-4} \!\!\! \sgn(m)q^{Q_2(\b{m})}.
\end{align*}
The last equality follows from observing that $(\frac{m - \ell + 5}{2})_{\ell-4} = 0$ for all $m$ that satisfy $|m| < \ell-3$ and $m \equiv \ell - 3 \pmod{2}$.

We can therefore express $Z(q)$ as a linear combination of series of the form
\begin{equation}
\label{Pochtwised}
\sum_{\substack{m \equiv g \pmod{D} \\ m \equiv \ell-3 \pmod{2}}} \left(\frac{m - \ell + 5}{2}\right)_{\ell-4} \sgn(m) q^{\frac12 am^2 + bm + c},
\end{equation}
where $a\in\N$, $b\in\Z$, and $c\in\C$. We now simplify the system to a single congruence as in \eqref{E:system}.

If $D$ is odd, then the system becomes $m \equiv t \pmod{2D}$ for some $t$. As before, we complete the square in the $q$-power to obtain that \eqref{Pochtwised} equals 
\begin{equation}\label{qseries}
q^{c - \frac{b^2}{2a}} \sum_{m \equiv t \pmod{2D}} \sgn(m) \left(\frac{m - \ell + 5}{2}\right)_{\ell-4} q^{\frac{a}{2} \left(m + \frac{b}{a}\right)^2}.
\end{equation}
Note that the Pochhammer symbol can be decomposed as
\begin{equation*}
\left(\frac{m - \ell + 5}{2}\right)_{\ell-4} = P_1 \left( \frac{a}{2} \left(m+\frac{b}{a}\right)^2\right)+\left(m+\frac{b}{a}\right)P_2 \left(\frac{a}{2} \left(m+\frac{b}{a}\right)^2\right),
\end{equation*}
where $P_1$ and $P_2$ are polynomials of degrees at most $\lfloor \frac{\ell-4}{2}\rfloor$ and $\lfloor \frac{\ell-5}{2}\rfloor$, respectively.
We can now write \eqref{qseries} as (ignoring the outside $q$-power)
\begin{align}
&f(q)+ P_1 (\mathcal{D}) \left(F_{at+b,aD}\left(2aD^2 \tau\right)\right) + 2D P_2(\mathcal{D}) \left(G_{at+b,aD}\left(2aD^2 \tau\right)\right),
\end{align}
where  $\mathcal{D} := q \frac{d}{dq}$ and
where $f(q)$ is a finite series in rational powers of $q$.

If $D$ is even and $g$ is odd, the calculations are analogous, with the only difference being that the modulus $2D$ is replaced by $D$ (which may be written as $2k$). Recalling Propositions \ref{quantum-12}, \ref{quantum-32} and \ref{quantum_set}, in all cases we can therefore express $Z(q)$ in the desired form.
\end{proof}

Observing that for $\ell=5$ the polynomials $P_1$ and $P_2$ in the proof of Theorem \ref{n-star} are constants, we obtain a special case.
\begin{corollary} \label{4-star-Z}
For any $4$-leg star graph,  $Z(q)$ is a linear combination of quantum modular forms (up to $q$-powers) of weights $\frac12$ or $\frac32$.
\end{corollary}

\begin{remark}
We are currently unable to prove that $Z(q)$ is quantum modular on a dense subset of $\mathbb{Q}$. The key issue is
that in Proposition \ref{quantum-32} we have three different types of quantum sets for the functions $G_{j,p}$. These three sets are pairwise disjoint, so we would have to prove that only one type can occur in a decomposition of $Z(q)$.
\end{remark}

\subsection{An example of $Z(q)$ with quantum set $\subsetneq \mathbb{Q}$ }

In this section we construct an example of a $4$-leg star graph such that $Z(q)$ is a linear combination of quantum modular forms (up to rational $q$-powers) that are not defined at all roots of unity. In particular, in this example $Z(q)$ is undefined at $q = 1$.

Consider the 4-leg star graph with matrix $M$
$$M=\left(\begin{smallmatrix} 3 & -1 & -1 & -1 & -1 \\ -1 & 3 & 0 & 0 & 0 \\ -1 & 0 & 3 & 0 & 0 \\ - 1 & 0 & 0 & 3 & 0 \\ - 1 & 0 & 0 & 0 & 3 \end{smallmatrix}\right),$$
which is clearly positive definite.
\begin{proposition}
\label{P:4star}
For the matrix $M$ given above,
\begin{equation}\label{ZFG}
Z(q) = - \frac23 q^{\frac23} F_{2,3}(30 \tau)+5q^{\frac23} G_{2,3}(30 \tau).
\end{equation}
Both $F_{2,3}(30\tau)$ and $G_{2,3}(30\tau)$ are quantum modular forms on
\begin{equation*}
S := \left\{ \frac{h}{k} \in \mathbb{Q}:  \gcd(h,k)=1, 3 | k, 4 \nmid k \right\},
\end{equation*}
but $Z(q)$ is not defined on all of $\Q$.
\end{proposition}

\begin{proof}
We have
\begin{equation*}
Z(q)= \textnormal{CT}_{\b w}\left(\frac{	 \Theta_M(\boldsymbol{w})\prod_{r=2}^5\left(w_r-w_r^{-1}\right)}{\left(w_1-w_1^{-1}\right)^2}\right).
\end{equation*}
As this example is relatively straightforward, we take a direct approach to evaluate $Z(q)$, rather than switching to \eqref{E:Theta2} and following the ``inverse matrix'' approach that appears in the proofs of Theorems \ref{T:Zqquantum} and \ref{n-star}. Plugging in the matrix, we have
\begin{align*}
Z(q) = \sum_{\b n\in\Z^5} q^{Q_1(\b n)} \operatorname{CT}_{w_1}\left(\frac{w_1^{3n_1-n_2-n_3-n_4-n_5}}{\left(w_1-w_1^{-1}\right)^2}\right) \prod_{r=2}^5 \operatorname{CT}_{w_r}
\left(\left(w_r-w_r^{-1}\right)w_r^{3n_r-n_1}\right).
\end{align*}
Applying \eqref{E:CTw-w^-1} for $w_2, w_3, w_4, \text{ and } w_5$ introduces the relations $3n_r-n_1=\pm 1$ for $2\leq r\leq 5$, which further implies that all $n_r$ must be equal. Writing the common value as $n$, we then have $n_1 = 3n \pm 1$. This yields that
\begin{align*}
Z(q) & = \sum_{ \pm } \sum_{n \in\Z} q^{Q_1(3n \pm 1, n, n, n, n)}\operatorname{CT}_{w}\left(\frac{w^{3(3n \pm 1) - 4n}}{\left(w-w^{-1}\right)^2}\right) \!\! = \!\sum_{\pm } \sum_{n \in \Z} q^{\frac12 \left(15n^2 \pm 10  n + 3\right)} \operatorname{CT}_{w}\left(\frac{w^{5n \pm 3}}{\left(w-w^{-1}\right)^2}\right).
\end{align*}
Corollary \ref{C:CTn} implies that
\begin{equation*}
{\rm CT}_w \left(\frac{w^m}{(w-w^{-1})^2}\right)
= \frac{m}{2} {\rm CT}_w \left(\frac{w^{m-1}}{w-w^{-1}}\right).
\end{equation*}
Thus
\begin{equation}\label{const}\begin{split}
Z(q) &= \frac12 \sum_{ \pm  } \sum_{n\in \Z} q^{\frac{15}{2} n^2 \pm 5  n + \frac32}(5n \pm 3) \operatorname{CT}_w \left(\frac{w^{5n \pm 3 - 1}}{w-w^{-1}}\right) \\
& = \frac12 \sum_{n\in \Z} q^{\frac{15}{2} n^2 + 5n + \frac32}(5n + 3) \operatorname{CT}_w \left(\frac{w^{5n + 3} - w^{-5n - 3}}{w(w-w^{-1})}\right),
\end{split}\end{equation}
where for the minus term we change $n\mapsto -n$. We now apply \eqref{E:CTmn}, which implies that
\begin{align*}
\operatorname{CT}_w \left(\frac{w^{5m + 3} - w^{-5m - 3}}{w(w-w^{-1})}\right)
= \delta_{|5m+3| \geq 2} \; \delta_{5m+3 \equiv 0 \pmod{2}} \; \sgn(5m+3)
= \delta_{m \equiv 1 \pmod{2}} \; \sgn(m).
\end{align*}
Thus \eqref{const} equals
\begin{equation*}
\frac{1}{2} \sum_{m \equiv 1 \pmod{2}} {\rm sgn}(m)(5m+3) q^{\frac{15}{2}m^2+5m+\frac{3}{2}}.
\end{equation*}
A direct computation then gives \eqref{ZFG}.

Note that Proposition \ref{quantum-32} implies that the series $G_{2,3}(3 \tau)$ has the quantum set
\begin{equation*}
\left\{ \frac{h}{k} \in \mathbb{Q}:  \gcd(h,k)=1, 3 | k ,  2 \nmid k \right\}.
\end{equation*}
We therefore have that
\begin{equation*}
\frac{1}{10}  \left\{ \frac{h}{k} \in \mathbb{Q}:  \gcd(h,k)=1, 3 | k ,  2 \nmid k \right\},
\end{equation*}
is a common quantum set for both $G_{2,3}(30 \tau)$ and  $F_{2,3}(30\tau)$. It is not hard to see that this set equals $S$.

To finish the proof we are left to show that $\lim_{t\to 0^+} G_{2, 3}(e^{-t})$ does not exist. This follows by following the well-known technique of Euler-MacLaurin summation for asymptotic series (see Proposition 6.5 of \cite{Za06}), and a short calculation verifies that the main term diverges. \qedhere

\end{proof}

\section{Conclusion and future work}

\begin{enumerate}[leftmargin=*, label=\textnormal{(\arabic*)}]
\item In this paper we demonstrate the quantum modularity of $Z(q)$ (up to a rational power of $q$) for every $3$-leg star graph.
In order to extend this result to general $n$-leg star graphs we need a better description of the individual quantum modular forms appearing in Theorem \ref{n-star}. As we discuss above, it is not clear whether the individual sums have a common quantum set. Note however that in all of the examples we have checked this is the case.

\item[(2)] More complicated (and interesting) examples of $Z(q)$ functions arise from non-star graphs such as (see \cite{GPPV}):
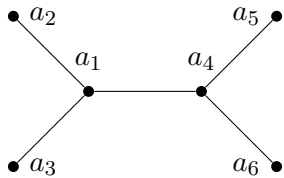
\begin{center}
			\begin{tikzpicture}
			\node[shape=circle,fill=black, scale = 0.4] (1) at (0,0) { };
			\node[shape=circle,fill=black, scale = 0.4] (2) at (-1,1) { };
			\node[shape=circle,fill=black, scale = 0.4] (3) at (-1,-1)  { };
			\node[shape=circle,fill=black, scale = 0.4] (4) at (1.5,0) { };
			\node[shape=circle,fill=black, scale = 0.4] (5) at (2.5,-1) { };
			\node[shape=circle,fill=black, scale = 0.4] (6) at (2.5,1) { };

						\node[draw=none] (B1) at (0,0.4) {$a_1$};
			\node[draw=none] (B2) at (-0.6,-1) {$a_3$};
			\node[draw=none] (B3) at (-0.6,1) {$a_2$};
			\node[draw=none] (B4) at (1.5, 0.4) {$a_4$};
			\node[draw=none] (B5) at (2.1,-1) {$a_6$};
			\node[draw=none] (B6) at (2.1,1) {$a_5$};

			\path [-] (1) edge node[left] {} (2);
			\path [-](1) edge node[left] {} (3);
			\path [-](1) edge node[left] {} (4);
			\path [-](4) edge node[left] {} (5);
			\path [-](4) edge node[left] {} (6);
			\end{tikzpicture}
		\end{center}

\end{enumerate}
The  series $\widehat{Z}_{0}(q)$ in these examples are reminiscent of double false theta functions studied in \cite{BKM}. It would be interesting to determine quantum modular properties of such series.


\begin{thebibliography}{99}



\bibitem{BCR}  K. Bringmann, T. Creutzig, and L. Rolen, {\it Negative index Jacobi forms and quantum modular forms}, Res. Math. Sci. {\bf 1} (1) (2014), 11.

\bibitem{BKM} K. Bringmann, J. Kaszian and A.Milas,  {\em Higher depth quantum modular forms, multiple Eichler integrals, and $\frak{sl}_3$ false theta functions} (2017), Res. Math. Sci., accepted for publication.

\bibitem{BM} K. Bringmann and A. Milas, {\it W-algebras, false theta functions and quantum modular forms, I}, Int. Math. Res. Not. {\bf 21} (2015), 11351--11387.

\bibitem{BM2} K. Bringmann and A. Milas,  {\it W-algebras, higher rank false theta functions, and quantum dimensions}, Selecta Math. {\bf 23} (2017), 1249--1278.

\bibitem{BHL} K. Bringmann, K. Hikami, and J. Lovejoy, {\em On the modularity of the unified WRT invariants of certain Seifert manifolds}, Adv. Appl. Math. {\bf 46} (2011), 86--93.

\bibitem{CCFGH} M. Cheng, S. Chun, F. Ferrari, S. Gukov, and S. Harrison, \emph{3d Modularity}, \texttt{arXiv:1809.10148}.

\bibitem{CM} T. Creutzig and A.Milas,  {\em Higher rank partial and false theta functions and representation theory},  Adv. Math. {\bf 314} (2017), 203--227.

\bibitem{FG} B. Feigin and S. Gukov, {\em VOA [M4]}, \texttt{arXiv:1806.02470.}	

\bibitem{Gukov} S. Gukov, Talk presented at the Banff workshop {\em Modular Forms and Quantum Invariants of Knots}, March 2018.

\bibitem{GPPV} S. Gukov, D. Pei, P. Putrov, and C. Vafa, {\em  BPS spectra and 3-manifold invariants}, \texttt{arXiv:1701.06567}.

\bibitem{H} K. Habiro, {\em A unified Witten-Reshetikhin-Turaev invariant for integral homology spheres}, Invent. Math. {\bf 171} (2008), 1--81.

\bibitem{Hi1} K. Hikami,  {\em Quantum Invariant, Modular Form, and Lattice Points}, IMRN {\bf 3} (2005), 121--154.

\bibitem{Hi2} K.  Hikami, {\em On the Quantum Invariant for the Spherical Seifert Manifold}, Commun. Math. Phys. {\bf 268} (2006), 285--319.


\bibitem{Hi3}  K. Hikami, {\em Quantum Invariants, Modular Forms, and Lattice Points II},	J. Math. Phys. {\bf 47} (2006), 102301.

\bibitem{Hikami}  K. Hikami and A. Kirillov, {\em Torus knot and minimal model}, Phys. Lett. B {\bf 575.3} (2003), 343--348.

\bibitem{Hikami2} K. Hikami and J. Lovejoy, {\em Torus knots and quantum modular forms}, Res. Math.  Sci. {\bf 2} (1) (2015), 2.

\bibitem{LZ} R. Lawrence and D. Zagier, \emph{Modular forms and quantum invariants of 3-manifolds}, Asian J. of Math. {\bf 3} (1999), 93--108.

\bibitem{OV} A. Onishchik and E. Vinberg, \emph{Lie groups and algebraic groups. Translated from the Russian and with a preface by D. A. Leites.} Springer Series in Soviet Mathematics, Springer-Verlag, Berlin, 1990.

\bibitem{RT} N. Reshetikhin and V. Turaev, {\em Invariants of 3-manifolds via link polynomials and quantum groups},  Invent. Math. {\bf 103} (1991), 547--597.

\bibitem{W} E. Witten, {\em Quantum field theory and the Jones polynomial},  Commun. Math. Phys. {\bf 121} (1989), 351--399.
	
\bibitem{Za1} D. Zagier, \begin{it}Vassiliev invariants and a strange identity related to the Dedekind eta-function, \end{it} Topology {\bf 40} (2001),  945--960.

\bibitem{Za06} D. Zagier, \emph{The Mellin transform and other useful analytic techniques.}
Appendix to E. Zeidler, Quantum Field Theory I: Basics in Mathematics and Physics. A Bridge Between Mathematicians and Physicists, Springer-Verlag, Berlin-Heidelberg-New York (2006), 305--323.

\bibitem{Za2}	D. Zagier, {\em Quantum modular forms}, Quanta of Math, {\bf 11} (2010), 659--675.
	
		
\end{thebibliography}
\end{document}